\expandafter\def\csname input@path\endcsname{%
    {sty/},
    {img/},
    {bib/}
}

\RequirePackage{xparse,expl3}
\documentclass[numbers]{vmsta2}
\usepackage{subcaption}
\volume{0}
\issue{0}
\pubyear{2026}
\articletype{research-article}

\newtheorem{thm}{Theorem}
\newtheorem{remark}{Remark}

\newtheorem{example}{Example}
\newtheorem{assumption}{Assumption}

\theoremstyle{definition}

\begin{document}
\begin{frontmatter}

\pretitle{Research Article}

\title{Multiscale Asymptotic Analysis of Kernel-Smoothed Solutions to Fractional Riesz-Bessel Equations with Random Initial Conditions}
\runtitle{Multiscale Asymptotics of Kernel-Smoothed Solutions to FRBE}

\begin{aug}
\author[a]{\inits{S.}\fnms{Shahid}~\snm{Khan}\ead[label=e1]{shahid.khan@latrobe.edu.au}\orcid{0009-0008-5402-1938}}

\thankstext[type=corresp,id=cor1]{Corresponding author.}
\author[a]{\inits{A.}\fnms{Andriy}~\snm{Olenko}\thanksref{cor1}\ead[label=e2]{a.olenko@latrobe.edu.au}\orcid{0000-0002-0917-7000}}

\address[a]{\institution{Department of Mathematical and Physical Sciences, La Trobe University}, \\
Melbourne, VIC 3086, \cny{Australia}}
\runauthor{S. Khan, A. Olenko}
\end{aug}

\dedicated{The paper is dedicated to the 75th anniversary of Professor Nikolai Leonenko.}

\begin{abstract}
This paper investigates fractional Riesz–Bessel equations with random initial conditions that exhibit either classical or cyclic long-range dependence. It studies zoom-in asymptotics for the corresponding solutions and establishes multiscaling limit theorems. It is known that for similar problems, non-degenerate multiscaling limits may not exist in general. The paper develops a kernel-smoothing approach for these equations and obtains non-degenerate limit fields under suitable normalisation and rescaling.
It proves that the kernel-smoothed solutions converge weakly to Gaussian random fields, which are non-stationary in both time and space. Their spectral representations and covariance functions are derived. The paper also analyses the regularity and dependence structure of the limit fields. In particular, under appropriate general assumptions on the smoothing kernel, the limits exhibit long-range dependence in time and short-range dependence in space. Numerical examples for the case of Matérn-type kernels are provided to illustrate the theoretical results.
\end{abstract}



\end{frontmatter}

\section{Introduction}
Fractional Riesz-Bessel equations extend classical kinetic models by using a composite fractional operator that combines the non-local behaviour of the fractional Laplacian with the regularising effect of the Bessel potential operator.
Such equations are widely used in the study of transport in disordered media, multifractal dynamics in financial systems, and high-frequency turbulence, see \cite{Gorenflo2020, Metzler1994}.
The work~\cite{Caputo1967} initiated the development of this area, while \cite{Oldham1974} and \cite{Podlubny1998} established detailed mathematical frameworks for fractional operators and integrals that are widely used today. Earlier publications considered these equations with random initial conditions given by Gaussian random fields \cite{Anh1999II, Anh2000}. For such models, the asymptotic behaviour of the solution is determined not only by the differential operators but also by the spectral structure of the initial condition.

The first probabilistic approaches to the heat equation with random initial conditions were introduced by De F\'eriet and Rosenblatt~\cite{De1956, Rosenblatt1968}. In recent years, there has been considerable interest in equations of this type with general random initial conditions. Several publications have demonstrated how the structure of these initial conditions influences the resulting solutions and their properties. The dependence structure of the initial random field is typically characterised by its spectral density: singularities at zero frequency correspond to classical long-range dependence, whereas singularities at nonzero frequencies are associated with cyclic long-memory behaviour \cite{Alghamdi2024}. Later work extended this analysis to broad classes of random environments and non-homogeneous settings~\cite{Becus1980, Uboe1995II}.  For classical diffusion equations, scaling limits of random solutions were derived in \cite{Albeverio1994, Leonenko1998}, while related methods were applied to Burgers-type equations in \cite{Leonenko1999, Leonenko1998a}. These studies established general principles for investigating asymptotics of rescaled random solution fields.

More recent research has focused on fractional operators and long-range dependence. Gay and Heyde \cite{Gay1990} showed that fractional Laplace operators generate random fields with long-range dependence. Further studies of stochastic heat equations with fractional Laplace operators were carried out in \cite{Angulo2000, Anh1999}. The results by Anh and Leonenko \cite{Anh1999II} proved that when the initial spectrum has a singularity at the origin, suitably normalised solutions converge to non-Gaussian limits. This established a direct link between fractional operators, singular spectral densities, and limit theorems. The analysis was further developed by renormalisation and homogenization methods for fractional-in-time/space diffusion equations with random input \cite{Anh2000, Anh2002}. These results belong to the non-central limit theory for strongly dependent random fields, see \cite{Dobrushin1979, Taqqu1979}. Related asymptotic problems have been studied for Airy and Korteweg--de Vries equations \cite{Beghin2000, Knopova2004}. Stochastic models on spheres and hyperbolic diffusion equations were analysed in \cite{Anh2021, Broadbrige2020}. Further results for diffusion in expanding space--time were recently obtained in \cite{Broadbridge2024}.

Contrary to the classical increasing domain asymptotic settings, multiscaling refers to asymptotic zoom-in regimes in which the normalisation and the space--time rescaling are used to study local behaviour.
For stochastic fractional diffusion equations, these settings lead to multiscaling limit theorems for their solutions, as well as the classification of the limit behaviour depending on the initial conditions~\cite{Alghamdi2025}. The analysis of fractional Riesz--Bessel equations with cyclic long-memory initial conditions in~\cite{Leonenko2024} demonstrated that such rescaled solutions converge to Gaussian random fields. Related scaling and homogenization questions have also been studied for time-fractional relativistic diffusion equations and for coupled parabolic and reaction--diffusion--wave systems \cite{Liu2010, Liu2015, Liu2018}. These works demonstrated how the limit behaviour depends on the fractional operator and the initial spectrum.

Kernel smoothing is often introduced to obtain improved behaviour and non-degenerate limits by controlling high-frequency contributions~\cite{Alghamdi2025}. In numerous applications, this approach functions as a filter applied to realisations of random processes or fields~\cite{Alodat2020, Anh2003}. For example, the work of~\cite{Zhang2022} considered kernel smoothing methods for stochastic partial differential equations, where solutions to the Fokker–Planck equation were approximated using adaptive Gaussian kernels. The approach was combined with stochastic filtering to estimate probability distributions from noisy observations. This established a direct link between kernel smoothing, types of considered equations and properties of the limits.

The existing literature (see~\cite{Alghamdi2025, Alghamdi2024, Alghamdi2025II} and the references therein) provides multiscaling theorems for several classes of equations with long-range and cyclically dependent initial conditions, including fractional and higher-order equations. However, the results in~\cite{Alghamdi2025, Alghamdi2025II} demonstrate that multiscaling limits do not exist in all cases and suggest using kernel smoothing to obtain non-degenerate asymptotic behaviour. To the best of our knowledge, no existing work has considered fractional Riesz–Bessel equations in conjunction with kernel smoothing and multiscale asymptotic analysis under general spectral settings. This study considers both classical and cyclic long-memory initial conditions, extending earlier results on the scaling limits of the corresponding solution fields. It also investigates the smoothness properties and dependence structures of these limit fields.

The paper is organised as follows. Section~\ref{defP2} provides the main definitions, notations, and some basic assumptions used throughout the paper. Section~\ref{MultiscalingFRF} derives limit theorems for filtered random fields with random initial conditions, for both classical and cyclic long-range dependence cases. It also presents numerical examples with realisations of the limit fields and their covariance structures.  H\"older continuity and long-range dependence properties of the obtained multiscaling limit fields are studied in Section~\ref{Holder}.  Section~\ref{conclusionP2} concludes the paper with a brief discussion and directions for future research.

All numerical computations and plotting in this paper were performed using the software R (version 4.6.0). The corresponding R code is freely available in the folder
”Research materials” from the website \url{https://sites.google.com/site/olenkoandriy/}.

\section{ Definitions and notations} \label{defP2}
This section provides the main notations and background material required in the following sections.

Standard notations such as $\overline{c}$ for the complex conjugate of $c$, $\mathcal{B}(\mathbb{R})$ for the $\sigma$-field of Borel sets on $\mathbb{R}$, $\Gamma(\cdot)$ for the Gamma function, and $\mathbf{1}_A(\cdot)$ for the indicator function of a set $A$ are used throughout. The symbol~$C$ denotes constants, which exact values are not important for this exposition and may vary, even within the same proof.

We denote the modified Bessel function of the second kind by $K_\nu(\cdot)$, see~\cite[p.~78]{watson1922}. The Mittag-Leffler function is denoted by $E_\beta(\cdot)$, where $0<\beta<1$. More details on the Mittag-Leffler function and its properties can be found in~\cite{Mainardi2014}.

For $s\in\mathbb{R}$, the Sobolev space $H^s(\mathbb{R})$ consists of all functions $f\in L^2(\mathbb{R})$ such that
$\|f\|_{H^s}^2
   : =
    \int_{\mathbb{R}}
    (1+|\lambda|^2)^s
    |\widehat f(\lambda)|^2\,d\lambda<+\infty.
$

Let $u(t,x)$ be a real-valued function of the two arguments $x \in \mathbb{R}^d$ and $t>0.$ The time derivative of order $\beta \in (0,1]$ is defined by
\begin{equation*} 
\frac{\partial^{\beta} u(t,x)}{\partial t^{\beta}} :=
\begin{cases}
\dfrac{\partial u}{\partial t}(t,x), & \text{if } \beta = 1, \\[10pt]
\left( \mathcal{D}_t^{\beta} u \right)(t,x), & \text{if } \beta \in (0,1),
\end{cases}
\end{equation*}
where
\[
\left( \mathcal{D}_t^{\beta} u \right)(t,x)
:= \frac{1}{\Gamma(1 - \beta)}
\left[
\frac{\partial}{\partial t} \int_0^t (t - \tau)^{-\beta} u(\tau, x)\, d\tau
- \frac{u(0,x)}{t^{\beta}}
\right],
\qquad t>0,
\]
is the regularised fractional derivative in the Caputo–Djrbashian sense \cite[(2.138)]{Podlubny1998}.

The publications \cite{Alghamdi2025II, Anh1999II} and \cite{Anh2000} studied fractional Riesz–Bessel equations (FRBE)
\begin{equation}\label{FRBE}
\frac{\partial^\beta u(t,x)}{\partial t^\beta} = -\mu (I - \Delta)^{\gamma/2} (-\Delta)^{\alpha/2} u(t,x), \quad t > 0,
\end{equation}
subject to a random initial condition
\begin{equation}\label{Icond}
u(0,x) = \xi(x),
\end{equation}
where $\alpha \geq 0, \, \gamma>0, \, \mu >0,$   $\xi(x)$, $x \in \mathbb{R}^d,$ is a zero-mean Gaussian random field, and $(-\Delta)^{\alpha/2}$ and $(I-\Delta)^{\gamma/2}$ are fractional inverse operators of the Riesz and Bessel potentials, respectively, see, for example, \cite{Anh1999II, Anh2000}.

The Cauchy problem (\ref{FRBE}) with the non-random initial condition given by the Dirac delta function, $u(0,x) = \delta(x),$  has a unique solution determined by $\widehat{G}(t,{x})$, $t>0$, ${x} \in \mathbb{R}^d,$ the Fourier transform with respect to the second argument of the fundamental solution (i.e., the Green function) of the non-random Cauchy problem.

The solution to the initial value problem (\ref{FRBE}) and (\ref{Icond}) in convolution form is defined~as
\begin{equation}
\label{convform}
u(t,x) = \int_{\mathbb{R}^n} G(t,x-y)\,u(0,y)\,dy,
\end{equation}
where the Green function $G(t,x)$ is specified by its Fourier transform.

In the literature this solution is called a Green-type solution or a mean-square solution, as it can be interpreted in the mean-square sense. For more details and justifications of this approach,
consult \cite{Anh2001a, Anh2003}.

For simplicity, this paper focuses on the one-dimensional case of $d=1.$ To make the notations consistent with multidimensional cases considered in the mentioned publications, the operator $\Delta=\partial^2/\partial x^2$ denotes the one-dimensional Laplacian.

Let $\xi(x)$, $x \in \mathbb{R}$, be a measurable stationary real-valued zero-mean Gaussian random process defined on a probability space $(\Omega, \mathcal{F}, P).$ Let $Z(\cdot)$ and $W(\cdot)$ denote, respectively, an orthogonal random measure and a Gaussian white-noise random measure on $\mathbb{R},$ with $\mathbb{E}|Z(d\lambda)|^2 = F(d\lambda),$ where $F(\cdot)$ is the spectral measure. In this paper, the random measures $Z(\cdot)$ and $W(\cdot)$ are symmetric, ensuring that the process $\xi(x)$ is real-valued.

The covariance function of $\xi(x)$ can be written as
\begin{equation}\label{P2cov}
 r(x):=\mathrm{Cov}(\xi(0), \xi(x)) = \int_{\mathbb{R}} e^{i  x \lambda } F(d\lambda),
\end{equation} where $F(\cdot)$ is the spectral measure.

If the spectral measure is absolutely continuous, it can be represented as
\[F(\Delta)=\int_{\Delta} f_{\xi}(\lambda)d\lambda,\quad \Delta\in \mathcal{B}(\mathbb R), \] where the function $f_{\xi}(\lambda), \lambda\in \mathbb R$ is called the spectral
density function of the stationary process $\xi(\cdot)$. As the process is real-valued, its spectral density is an even function.

Then, the following spectral representation of the random process $\xi(x)$ holds true
\[
\xi(x) = \int_{\mathbb {R}}e^{i \lambda x } Z(d\lambda)=\int_{\mathbb {R}}e^{i \lambda x } \sqrt{f_{\xi}(\lambda)}W(d\lambda).
\]

The Fourier transform of the Green function $G(t,x)$ is equal to, see  \cite{Anh2001a, Anh2003},
\begin{equation*}
\widehat{G}(t,\lambda)
= E_{\beta}\!\left(
-\mu t^{\beta} |\lambda|^{\alpha} \left( 1 + |\lambda|^2 \right)^{\gamma/2}
\right),
\end{equation*}
 where $E_\beta(\cdot)$ is the Mittag-Leffler function defined as
\[
E_\beta(s) := \sum_{k=0}^{\infty} \frac{s^k}{\Gamma(1 + \beta k)}, \qquad s \in \mathbb{R},\ 0 < \beta < 1.
\]
 For the negative values of its argument, it satisfies the inequality
\begin{equation}\label{upperch5}
   \frac{1}{1 + \Gamma(1-\beta)s} \leq E_\beta(-s) \leq\frac{1}{1 + \frac{s}{\Gamma(1+\beta)}}, \quad s\geq 0.
\end{equation}

The solution~(\ref{convform}) of the initial value problem (\ref{FRBE}) and (\ref{Icond}) admits the following spectral representation, see \cite{Anh2003},
\begin{equation}\label{P2sol}
    u(t,x) = \int_{\mathbb{R}} e^{ix \lambda} E_\beta(-\mu |\lambda|^{\alpha} (1+|\lambda|^2)^{\gamma/2} t^\beta) Z(d\lambda).
\end{equation}

The covariance function of the solution field $u(t, x)$ is
\begin{align*}
    \mathrm{Cov}(u(t, x), u(t', x')) &= \int_{\mathbb{R}} e^{i(x-x') \lambda}
 E_\beta(-\mu |\lambda|^{\alpha} (1+|\lambda|^2)^{\gamma/2} t^\beta)\\ &\quad \quad\times  E_\beta(-\mu |\lambda|^{\alpha} (1+|\lambda|^2)^{\gamma/2} (t')^\beta)  F(d\lambda).
\end{align*}

To further specify the class of Gaussian random processes $\xi(x),$ $x \in \mathbb R,$ used as the initial condition~\eqref{Icond}, the following assumption is used, see \cite{Alghamdi2024} for more details.

\begin{assumption}\label{Asumptionmain}
    The covariance function (\ref{P2cov}) has the form
\begin{align}\label{P2cov1}
    &r(x)=\sum_{j=0}^{n}\frac{\cos(w_{j}x)}{(1+x^{2})^{\kappa_j/2}}A_{j},\quad x\in\mathbb R,
\end{align}
where $\sum_{j=0}^{n}A_{j}=1,$  $A_j\ge 0,$ $\kappa_j\in(0,1),$ $w_{j}>0$ (except $w_{0}=0$), $ j=0,...,n.$
\end{assumption}
Let us define the next constants
\[C(\bar{\kappa},\bar{w},\bar{A}):=\sum_{j=1}^{n}
c_1(\kappa_j)A_j
K_{\frac{\kappa_j-1}{2}}(|w_j|)
|w_j|^{\frac{\kappa_j-1}{2}},\]
\[
C(\kappa_0,A_0)
:=
\frac{A_0}
{\Gamma\!\left({\kappa_0}\right)\cos(\kappa_0\pi/2)},
\]
where $c_1(\kappa_j):=2^{\mathbf{1}_{\{0\}}(j)}{2^{\frac{1-\kappa_j}{2}}}/\left({\sqrt{\pi}{\Gamma\left({\kappa_j}/{2} \right)}}\right),$  $\bar{\kappa}=( \kappa_1,...,\kappa_n),$ $\bar{w}=(w_1,...,w_n),$ and $\bar{A}= (A_1,..., A_n)$.

The covariance function in \eqref{P2cov1} is non-integrable and exhibits an oscillating behaviour, which corresponds to the cyclic long-range dependence scenario.
It follows from (\ref{P2cov1}) that the corresponding spectral density has the representation
\begin{align}
f_{\xi}(\lambda)
& :=\sum_{j=1}^{n}\frac{c_1(\kappa_j)}{2}A_{j}
\Big(
K_{\frac{\kappa_j-1}{2}}\left(|\lambda+w_j|\right)|\lambda+w_j|^{\frac{\kappa_j-1}{2}}\\
&\quad\quad +
K_{\frac{\kappa_j-1}{2}}\left(|\lambda-w_j|\right)|\lambda-w_j|^{\frac{\kappa_j-1}{2}}
\Big) +\,\frac{c_{1}(\kappa_0)}{2}A_{0}
K_{\frac{\kappa_{0}-1}{2}}(|\lambda|)|\lambda|^{\frac{\kappa_{0}-1}{2}},\nonumber
\end{align}
where the modified Bessel function of the second kind  $K_\nu(\cdot)$ is defined by \[K_\nu(z):=\frac{1}{2}\int_{0}^{+\infty}s^{\nu-1}\exp\left(-\frac{1}{2}\left(s+\frac{1}{s} \right)z\right)ds,\quad z\geq 0,\quad \nu\in \mathbb {R}.\]

\section{Multiscaling limit theorems for filtered random fields} \label{MultiscalingFRF}

The publications~\cite{Alghamdi2025, Alghamdi2024} showed that, for the Cauchy problem with random initial conditions, there are cases where multiscaling limits do not exist. They also proposed kernel smoothing as a way to obtain non-degenerate limits in the case of higher-order heat equations. The following results apply this approach to the fractional Riesz-Bessel equations studied in~\cite{Alghamdi2025II}.

Consider a spatial kernel $g(x_1,x),$ $x_1, x\in \mathbb R,$ and define the centered spatially averaged random field
\begin{equation}\label{kernsmooth}
U_\varepsilon^{g}(t,x)
:=
\varepsilon^{\rho_1}
\left(
\int_{\mathbb{R}} g(x_1,x)\,
u\!\left(\frac{t}{\varepsilon^{\rho_2}},\frac{x_1}{\varepsilon^{\rho_3}}\right)\,dx_1
-
\mathbb E \int_{\mathbb{R}} g(x_1,x)\,
u\!\left(\frac{t}{\varepsilon^{\rho_2}},\frac{x_1}{\varepsilon^{\rho_3}}\right)\,dx_1
\right),
\end{equation}
where $\rho_i \in \mathbb R,$ $i=1,2,3,$ are some scaling parameters. As the initial condition is given by a zero-mean Gaussian random field, the field $u(t,x)$ is also zero-mean, and the expectation term in (\ref{kernsmooth}) vanishes.

We will use real-valued spatial kernel functions $g(x_1,x),$ $x_1, x\in\mathbb{R}$, such that for all $x$ it holds $g(\cdot,x) \in L_2(\mathbb{R})$. Their Fourier transforms with respect to the first argument are given by
\[
\widehat{g}(\lambda,x)
:=
\int_{\mathbb{R}} e^{-i\lambda x_1} g(x_1,x)\,dx_1.
\]
Since the Fourier transform of these kernel functions is taken with respect to the first variable $z$, we will use the notation $\|g(z,x)\|_{H_z^{s}}$ for the corresponding Sobolev norm in the $z$-variable.

\begin{thm}\label{ThmP2.1}
Consider the random field $u(t,x)$, $t>0$, $x\in\mathbb{R}$, defined by \eqref{FRBE} and the random initial condition given by \eqref{Icond}, satisfying Assumption~{\rm\ref{Asumptionmain}} with $A_0=0$.

For  $\rho_1=-{\rho_3}/{2}$, $\rho_2=\alpha\rho_3/\beta$, and  $\rho_3>0$, when $\varepsilon\to 0$, the finite-dimensional distributions of $U_\varepsilon^{g}(t,x)$ converge weakly to those of a zero-mean Gaussian random field~$U_0^{g}(t,x)$ given by
\begin{equation} \label{TH23.1A0=0}
U_0^{g}(t,x)
:=
\sqrt{
C(\bar{\kappa},\bar{w},\bar{A})}
\int_{\mathbb{R}}
\widehat{g}(\lambda,x)\,
E_\beta\!\left(-\mu t^\beta |\lambda|^\alpha\right)\,
W(d\lambda).
\end{equation}

The limit random field $U_0^{g}(t,x)$ has the covariance function
\begin{align}
\operatorname{Cov}\bigl(U_0^g(t,x),U_0^g(t',x')\bigr)
&=
C(\bar{\kappa},\bar{w},\bar{A})
\int_{\mathbb{R}}
\widehat{g}(\lambda,x)\overline{\widehat{g}(\lambda,x')}\nonumber \\ &\quad \quad\times
E_\beta(-\mu t^\beta |\lambda|^\alpha)
E_\beta(-\mu (t')^\beta |\lambda|^\alpha)\,d\lambda.
 \label{CovTH3.1a0=0}
\end{align}
\end{thm}

\begin{proof}
Using the spectral representation~\eqref{P2sol} of the solution and substituting $t$ with $t/\varepsilon^{\rho_2}$  and $x$ with $x_1/\varepsilon^{\rho_3}$, one gets
\begin{equation} \label{utx}
u\!\left(\frac{t}{\varepsilon^{\rho_2}},\frac{x_1}{\varepsilon^{\rho_3}}\right)
=
\int_{\mathbb R}
e^{i\frac{x_1}{\varepsilon^{\rho_3}}\lambda}
E_\beta\!\left(
-\mu |\lambda|^\alpha (1+|\lambda|^2)^{\gamma/2}
\left(\frac{t}{\varepsilon^{\rho_2}}\right)^\beta
\right)
\sqrt{f(\lambda)}\,W(d\lambda).
\end{equation}

The  application of kernel smoothing~\eqref{kernsmooth} to~\eqref{utx} results in
\begin{align*}
U_\varepsilon^g(t,x)
&=
\varepsilon^{\rho_1}
\int_{\mathbb R}
\int_{\mathbb R}
g(x_1,x)
e^{i\frac{x_1}{\varepsilon^{\rho_3}}\lambda}
E_\beta\!\left(
-\mu |\lambda|^\alpha (1+|\lambda|^2)^{\gamma/2}
t^\beta \varepsilon^{-\rho_2\beta}
\right)\\ &\quad \quad \quad \quad\quad \quad\times \sqrt{f(\lambda)}\,W(d\lambda)\,dx_1.
\end{align*}

As the Mittag-Leffler function is bounded and
$g(x_1, x)\sqrt{f(\lambda)} \in L_2(\mathbb{R}^2)$ for each fixed $x$, the stochastic Fubini theorem can be applied,
allowing to interchange the order of integration. Then, applying the change of variable $\tilde{\lambda}:=\lambda/\varepsilon^{\rho_3}$ and using the scaling property of Gaussian white noise
$W(\varepsilon^{\rho_3}d\tilde{\lambda})
\stackrel{d}{=}
\varepsilon^{\rho_3/2}W(d\tilde{\lambda}),$ we obtain
\begin{align*}
U_\varepsilon^g(t,x)
&=
\varepsilon^{\rho_1+\rho_3/2}
\int_{\mathbb R}
\widehat g(\tilde{\lambda},x)
E_\beta\!\left(
-\mu |\tilde{\lambda}|^\alpha
\left(1+|\tilde{\lambda}|^2\varepsilon^{2\rho_3}\right)^{\gamma/2}
t^\beta \varepsilon^{\alpha\rho_3-\rho_2\beta}
\right) \\ &\quad \quad\quad \quad\quad \quad\times
\sqrt{f(\tilde{\lambda}\varepsilon^{\rho_3})}\,W(d\tilde{\lambda}).
\end{align*}

The values of the scaling parameters for which the normalisation factor is constant and the Mittag-Leffler function remains nondegenerate in the limit are
$\rho_1=-{\rho_3}/{2},$ $\rho_3>0,$ and $\alpha\rho_3-\rho_2\beta=0.$ In this case
\[
U_\varepsilon^g(t,x)
=
\int_{\mathbb R}
\widehat g(\tilde{\lambda},x)
E_\beta\!\left(
-\mu|\tilde{\lambda}|^\alpha \left(1+|\tilde{\lambda}|^2\varepsilon^{2\rho_3}\right)^{\gamma/2} t^\beta
\right)
\sqrt{f(\tilde{\lambda}\varepsilon^{\rho_3})}\,W(d\tilde{\lambda}).
\]

Since $\rho_3>0$, if $\varepsilon\to 0,$ one has
\[
E_\beta\!\left(
-\mu|\tilde{\lambda}|^\alpha \left(1+|\tilde{\lambda}|^2\varepsilon^{2\rho_3}\right)^{\gamma/2} t^\beta
\right)
\longrightarrow
E_\beta\!\left(-\mu t^\beta |\tilde{\lambda}|^\alpha\right).
\]

It suggests that the limit random field is given by
\[
U_0^g(t,x)
=
\sqrt{C(\bar{\kappa},\bar{w},\bar{A})}
\int_{\mathbb R}
\widehat g(\lambda,x)
E_\beta\!\left(-\mu t^\beta |\lambda|^\alpha\right)
W(d\lambda).
\]

To justify the convergence, consider $R(t,x):=\mathbb{E}\left(U_\varepsilon^g(t,x)-U_0^g(t,x)\right)^2.$

Let us denote
\[
Q_\varepsilon(\tilde{\lambda})
:=
\sum_{j=1}^{n}\frac{c_1(\kappa_j)}{2}A_j
\Big(
K_{\frac{\kappa_j-1}{2}}\bigl(|\tilde{\lambda}\varepsilon^{\rho_3}+w_j|\bigr)
|\tilde{\lambda}\varepsilon^{\rho_3}+w_j|^{\frac{\kappa_j-1}{2}}\]
\[
+\,
K_{\frac{\kappa_j-1}{2}}\bigl(|\tilde{\lambda}\varepsilon^{\rho_3}-w_j|\bigr)
|\tilde{\lambda}\varepsilon^{\rho_3}-w_j|^{\frac{\kappa_j-1}{2}}
\Big).
\]

By the It\^o isometry,
\[
R_\varepsilon(t,x)
=
\int_{\mathbb R}
|\widehat g(\tilde{\lambda},x)|^2
E_\beta^2\!\left(-\mu t^\beta |\tilde{\lambda}|^\alpha\right)
\left(
\frac{
E_\beta\!\left(
-\mu |\tilde{\lambda}|^\alpha
\left(1+|\tilde{\lambda}|^2\varepsilon^{2\rho_3}\right)^{\gamma/2}
t^\beta
\right)
}{
E_\beta\!\left(-\mu t^\beta |\tilde{\lambda}|^\alpha\right)
}\right.\]
\[\left. \times\,
\sqrt{Q_\varepsilon(\tilde{\lambda})}
-
\sqrt{C(\bar{\kappa},\bar{w},\bar{A})}
\right)^2
d\tilde{\lambda}.
\]

Since
\[\frac{
E_\beta\!\left(
-\mu |\tilde{\lambda}|^\alpha
\left(1+|\tilde{\lambda}|^2\varepsilon^{2\rho_3}\right)^{\gamma/2}
t^\beta
\right)
}{
E_\beta\!\left(-\mu t^\beta |\tilde{\lambda}|^\alpha\right)
}
\to 1
\quad \text{and} \quad
Q_\varepsilon(\tilde{\lambda}) \to C(\bar{\kappa},\bar{w},\bar{A}),\]
when $\varepsilon\to 0,$ the integrand converges pointwise to zero.

By the properties of the Fourier transforms, $\widehat g(\cdot,x)\in L_2(\mathbb R)$ and the integrand is dominated by an integrable function. Hence, one can apply the generalised Lebesgue dominated convergence theorem, and the remaining part of the proof uses the same arguments as in the proof of \cite[Theorem 3.1]{Alghamdi2025II}.

We have shown that $R_{\varepsilon}(t,x)\to 0$ as $\varepsilon\to 0$. Hence, for arbitrary $a_1,\dots,a_m\in\mathbb{R}$,
\[
\lim_{\varepsilon\to 0}
\mathbb{E}
\left(
\sum_{j=1}^m a_j
\bigl(
U_\varepsilon^g(t_j,x_j)-U_0^g(t_j,x_j)
\bigr)
\right)^2
=0.
\]
Therefore, the Cram\'{e}r-Wold theorem gives the convergence of finite-dimensional distributions.

Because $\widehat g(\cdot,x)\in L_2(\mathbb R)$ and, by (\ref{upperch5}), the function~$E_{\beta}^{2}(\cdot)$ is bounded, the following integral is finite for all $\alpha \ge 0$
\[\int_{\mathbb R}|\widehat g(\lambda,x)|^2 E_{\beta}^{2}(-\mu t^{\beta}|\lambda|^{\alpha})d\lambda<\infty.
\]
Therefore, compared to \cite[Theorem 3.1]{Alghamdi2025II}, the condition $\alpha>1/2$ is not required.

The covariance representation~(\ref{CovTH3.1a0=0}) follows directly from~(\ref{TH23.1A0=0}) by the orthogonality property of the random measure $W(\cdot)$.
\end{proof}
\begin{remark}\label{Rmatern}
 The Mat\'ern kernel is a widely used class of continuous kernels in many recent applications, see, for example, \cite{ Leonenko2022, Porcu2024} and the references therein.
   For the parameters $\nu > 0$ and $a > 0$ it is defined by
    \begin{equation}\label{mkernel}
        h(x)
       : = \frac{ \left( a{|x|} \right)^\nu
        K_\nu\!\left( a{|x|} \right)}{2^{\nu-1}\Gamma(\nu)},
    \end{equation}
    where $K_\nu(\cdot)$ is the modified Bessel function of the second kind.

Its Fourier transform is given by
    \begin{equation}\label{FTmkernel}
        \widehat{h}(\lambda)
        = \frac{2\sqrt{\pi}a^{2\nu}\,\Gamma\!\left(\nu + \tfrac{1}{2}\right)}{\Gamma(\nu)} \,
        \left( {a^2} + \lambda^2 \right)^{-\left(\nu + \tfrac{1}{2}\right)}.
    \end{equation}

Some important special cases of the simplified expressions of~(\ref{mkernel}) and~(\ref{FTmkernel}), that are frequently used in applications, are:
    \begin{itemize}
        \item $\nu = \tfrac{1}{2}$:
$ h(x) = e^{-a|x| },$ $\widehat{h}(\lambda) = \frac{2a}{a ^{2} + \lambda^2}
$;
 \item $\nu = \tfrac{3}{2}$:
$h(x) = \left(1 + a{|x|}\right)e^{-a|x| },
$ $\widehat{h}(\lambda) = \frac{4a^3}{(a^2+\lambda^2)^2}$;
  \item $\nu = \tfrac{5}{2}$:
$h(x) = \left(1 + a{|x|} + \frac{(ax)^2}{3}\right)e^{-a|x|},
$ $\widehat{h}(\lambda) = \frac{16a^5}{3(a^2+\lambda^2)^3}.$
\end{itemize}
When the Mat\'ern kernels are used as covariance functions, the parameter $\nu$
controls the smoothness of the corresponding stochastic processes.
The case $\nu=1/2$ produces rougher trajectories, whereas the cases
$\nu=3/2$ and $\nu=5/2$ yield increasingly smoother sample paths.

Consider the kernel function used in  (\ref{kernsmooth}) of the form \[g(x_1,x):=h(x_1-x)= \frac{\left( {a} {|x_1-x|}\right)^\nu
K_\nu\!\left( {a} {|x_1-x|} \right)}{2^{\nu-1}\Gamma(\nu)}
.\]
 Then, for any $x \in \mathbb{R}$, by the properties of the Fourier transform,
\begin{equation} \label{MFourier}\widehat{g}(\lambda,x)=\widehat{h}(\lambda-x)= e^{i\lambda x}\,\widehat{h}(\lambda)=e^{i\lambda x}\frac{2\sqrt{\pi}a^{2\nu}\,\Gamma\!\left(\nu + \tfrac{1}{2}\right)}{\Gamma(\nu)} \,
\left( {a^2} + \lambda^2 \right)^{-\left(\nu + \tfrac{1}{2}\right)}.
\end{equation}

Since the fields are real-valued and the integrand in~(\ref{TH23.1A0=0}) is symmetric with respect to~$\lambda$, the complex exponential term $e^{i\lambda x}$ in the expression above for $\widehat{g}(\lambda,x)$ can be replaced by $\cos(\lambda x)$ without altering the limit field and its covariance structure.
\end{remark}

\begin{example} \label{numexP21}
This numerical example illustrates Theorem~\ref{ThmP2.1}. The parameter vectors were chosen as $\bar{\kappa}=(0.2,0.6,0.8)$ and $\bar{w}=(0.8,1.2,2.0)$, with weights $\bar{A}=(0.4,0.35,0.25)$ satisfying the normalization condition $\sum_{j=1}^{n}A_j=1$, since $A_0=0$. The parameters of the FRBE were selected as $\alpha=1$, $\beta=0.5$, and $\mu=1$.

The random field $U_0^g(t,x)$ was defined by using its spectral representation~\eqref{TH23.1A0=0}, where the stochastic integral was approximated by a Riemann-type sum over a uniform symmetric grid in the frequency domain:
\[
U_0^g(t,x) \approx
\sqrt{C(\bar{\kappa},\bar{w},\bar{A})} \sum_{j=-N}^{N}
\widehat{g}(\lambda_j,x)\,
E_{\beta}\!\left(-\mu t^{\beta}|\lambda_j|^{\alpha}\right)
\, W(\Delta \lambda_j),
\]
with $\lambda_j = j\Delta,$ $j=-N,\ldots,N.$

The selected discretisation step $\Delta=0.01$ and the truncation level $N=1000$
correspond to the approximation of the integrals over the interval
$[-10,10]$. This range was chosen because, outside this interval, the integrants are sufficiently small and make a negligible contribution to the integrals. The Gaussian increments
$W(\Delta\lambda_j)$ are taken to be independent normally distributed random
variables,
$W(\Delta\lambda_j)\sim N(0,\Delta).
$
They are set symmetrically with respect to the origin due to the real-valued initial-condition random field. The random field
$U_0^g(t,x)$ is then simulated in the spatio-temporal domain
$x\in[0,40],$ $t\in[0,2].$

For $a=1$, and for the Mat\'ern smoothness parameters $\nu=1/2$ and
$\nu=3/2$, the kernels in Remark~\ref{Rmatern} reduce, respectively, to
\begin{equation} \label{exampv=1/2}
g(x_1,x)=e^{-|x-x_1|},
\qquad
\widehat{g}(\lambda,x)
=
\frac{2\cos(\lambda x)}{1+\lambda^2},
\end{equation}
and
\begin{equation} \label{examplev=3/2}
g(x_1,x)
=
\left(1+|x-x_1|\right)e^{-|x-x_1|},
\qquad
\widehat{g}(\lambda,x)
=
\frac{4\cos(\lambda x)}{(1+\lambda^2)^2}.
\end{equation}

Figure~\ref{fig:P2th3.1a} presents a realisation of $U_0^g(t,x)$ for
$\nu=1/2$. Similar plots were obtained for $\nu=3/2$ and $\nu=5/2$, which exhibited smoothed realisations,
but they are omitted for brevity. The field varies mainly along the spatial
direction $x$, exhibiting repeated ridge-like structures. For small values
of $t$, the field exhibits oscillations of larger magnitudes.
As $t$ increases, the Mittag--Leffler factor
progressively damps these oscillations, leading to a gradual smoothing of the
same spatial pattern.

As the limit random field is non-stationary in both space and time, several plots are provided to illustrate the behaviour of its covariance function at different spatial and temporal locations. The covariance function
$\mathrm{Cov}\bigl(U_0^g(t,x),U_0^g(t',x')\bigr)
$
shown in Figure~\ref{fig:P2th3.1b} is obtained by approximating the integral
in~\eqref{CovTH3.1a0=0} at the fixed point $(t,x)=(1,20)$. As
expected, the covariance attains its maximum at this point and
decreases as either the temporal or spatial separation increases, although the decay is considerably slower in time.

Figure~\ref{fig:P2th3.1c} presents examples of spatial covariance functions
plotted for $x'\in[-30,30]$, with the temporal variables fixed at
$t=t'=1$. Two spatial reference locations, $x=1$ and $x=5$, and two
Mat\'ern smoothness parameters, $\nu=1/2$ and $\nu=3/2$, were used.
In both cases, the covariance attains its largest values near the selected
point $x$ and decreases with~$x'$ as the spatial separation between the two points
increases. The temporal covariance shown in Figure~\ref{fig:P2th3.1d} was obtained by
fixing $t=0.1$ and plotting the covariance as a function of
$t'\in[0,100]$. The covariance takes high values when the temporal separation
between $t$ and $t'$ is small and gradually decreases as the separation
increases. The slow temporal decay of the covariance suggests possible
long-memory behaviour. As expected, the covariance corresponding
to the closer spatial locations $(x,x')=(1,1)$ is larger than that for the
more separated locations $(x,x')=(1,5)$.\\[-5mm]
\begin{figure}[htbp]
    \centering
    \begin{subfigure}{0.49\textwidth}
        \centering
        \includegraphics[width=\linewidth, height=0.8\linewidth, trim=15mm  10mm 5mm 35mm,clip]{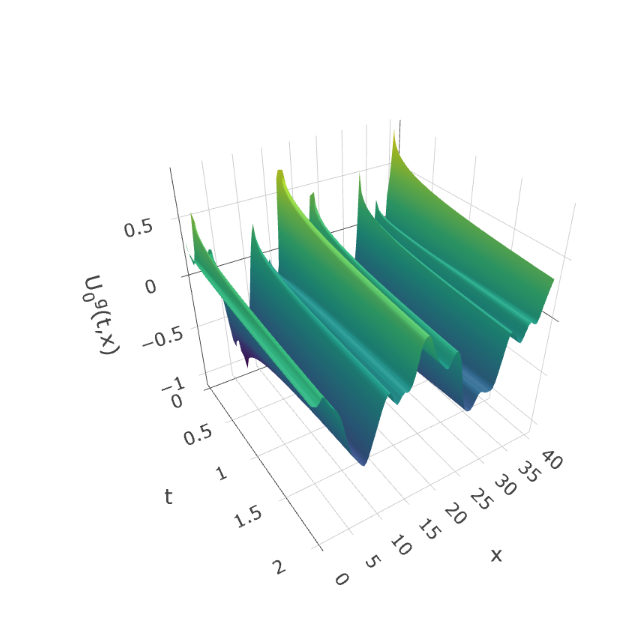}
        \caption{Realization of $U_0^g(t,x)$}
        \label{fig:P2th3.1a}
    \end{subfigure}\hfill
    \begin{subfigure}{0.49\textwidth}
        \centering        \includegraphics[width=\linewidth,height=0.8\linewidth, trim=10mm  10mm 5mm 35mm,clip]{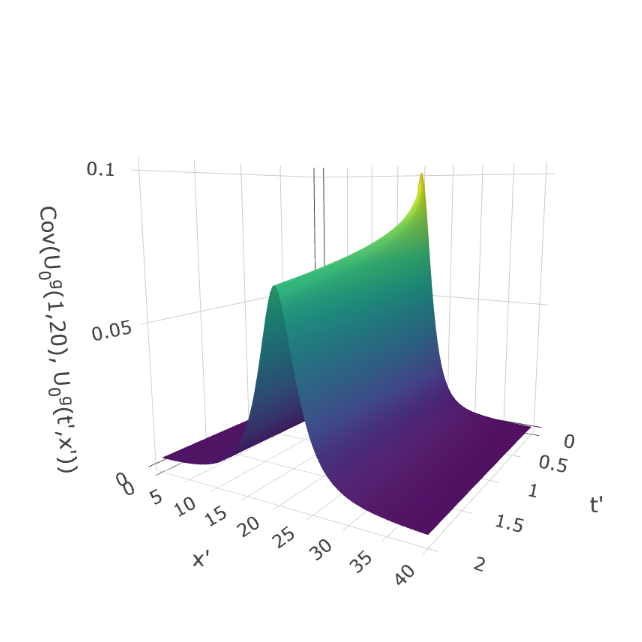}
        \caption{Covariance function of $U_0^g(t,x)$}
        \label{fig:P2th3.1b}
    \end{subfigure}
    \begin{subfigure}{0.48\textwidth}
        \centering
        \includegraphics[width=\linewidth, height=0.8\linewidth]{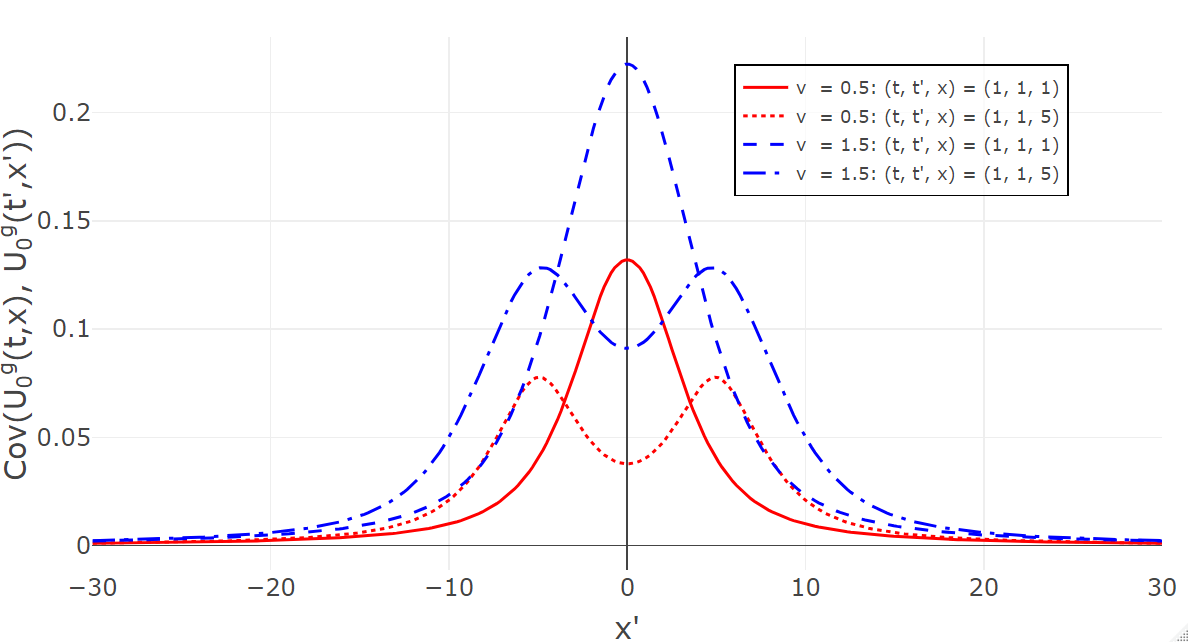}
        \caption{Examples of spatial covariance functions, for fixed $t$, $t'$ and $\nu=0.5$ and $1.5$}
        \label{fig:P2th3.1c}
    \end{subfigure}\hfill
    \begin{subfigure}{0.48\textwidth}
        \centering        \includegraphics[width=\linewidth, height=0.8\linewidth]{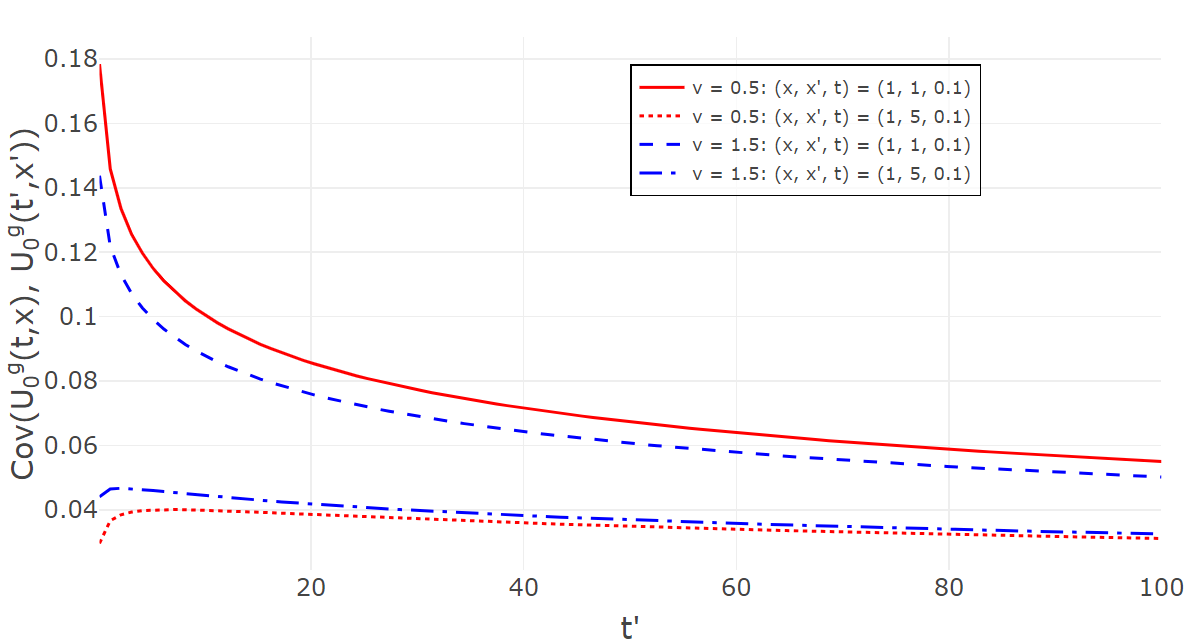}
        \caption{Examples of temporal covariance functions, for fixed $x,$ $x'$ and $\nu=0.5$ and $1.5$}
        \label{fig:P2th3.1d}
    \end{subfigure}
    \caption{Example of the limit field $U_0^g(t,x)$ from Theorem~\ref{ThmP2.1}}
    \label{fig:P2th3.1}
\end{figure}
   \end{example}
\begin{thm}\label{ThmP2.2}
Consider the random field $u(t,x)$, $t>0$, $x\in\mathbb{R}$, defined by \eqref{FRBE} and the random initial condition given by \eqref{Icond}, satisfying Assumption~{\rm\ref{Asumptionmain}} with $A_0 \ne 0$.

For the scaling parameters $\rho_1=-{\kappa_0\rho_3}/{2},$ $\rho_2=\alpha\rho_3/\beta,$ and $\rho_3>0$, when $\varepsilon\to 0$, the finite-dimensional distributions of $U_\varepsilon^{g}(t,x)$ converge weakly to those of a zero-mean Gaussian random field $\tilde{U}_0^{g}(t,x)$ given by
\begin{equation} \label{RTH3.2A0>0}
\tilde{U}_0^{g}(t,x)
:=
\sqrt{C(\kappa_0,A_0)}
\int_{\mathbb{R}}
\widehat{g}(\lambda,x)\,
E_\beta\!\left(-\mu t^\beta |\lambda|^\alpha\right)
|\lambda|^{\frac{\kappa_0-1}{2}}
\,W(d\lambda).
\end{equation}

The limit random field $\tilde{U}_0^{g}(t,x)$ has the covariance function
\begin{align}
\operatorname{Cov}\bigl(\tilde{U}_0^g(t,x),\tilde{U}_0^g(t',x')\bigr)
&=
 C(\kappa_0,A_0)
\int_{\mathbb{R}}
\widehat{g}(\lambda,x)\overline{\widehat{g}(\lambda,x')}\nonumber\\
&\quad \quad \times
E_\beta(-\mu t^\beta |\lambda|^\alpha)
E_\beta(-\mu (t')^\beta |\lambda|^\alpha)
|\lambda|^{\kappa_0-1}\,d\lambda. \label{covTH3.2A0>0}
\end{align}
\end{thm}

\begin{proof}

Using the same arguments as in the proof of Theorem~\ref{ThmP2.1}, we get
\begin{align*}
U_\varepsilon^g(t,x)
&=
\varepsilon^{\rho_1+\rho_3/2}
\int_{\mathbb R}
\widehat g(\tilde{\lambda},x)
E_\beta\!\left(
-\mu |\tilde{\lambda}|^\alpha
\left(1+|\tilde{\lambda}|^2\varepsilon^{2\rho_3}\right)^{\gamma/2}
t^\beta \varepsilon^{\alpha\rho_3-\rho_2\beta}
\right) \\ &\quad \quad \quad \quad \quad \quad\times
\sqrt{f(\tilde{\lambda}\varepsilon^{\rho_3})}\,W(d\tilde{\lambda}).
\end{align*}

Since $A_0>0$, the spectral density satisfies
\[
f(\lambda)\sim C(\kappa_0,A_0)\,|\lambda|^{\kappa_0-1},
\qquad \lambda\to0.
\]

Therefore, to obtain a non-trivial limit, one can choose
\[\alpha\rho_3-\rho_2\beta=0, \quad \rho_1=-\frac{\kappa_0\rho_3}{2},\quad \rho_3>0.
\]
Then, the smoothed field can be written as
\[
U_\varepsilon^g(t,x)
=
\int_{\mathbb R}
\widehat g(\tilde\lambda,x)
E_\beta\!\left(
-\mu |\tilde\lambda|^\alpha
(1+|\tilde\lambda|^2\varepsilon^{2\rho_3})^{\gamma/2}
t^\beta
\right)
\sqrt{\frac{f(\varepsilon^{\rho_3}\tilde\lambda)}
{\varepsilon^{\rho_3(\kappa_0-1)}}}
\,W(d\tilde\lambda).
\]

Since $\rho_3>0$, when $\varepsilon \to 0,$ the integrand converges pointwise to
\[
\sqrt{C(\kappa_0,A_0)}\,
\widehat g(\tilde\lambda,x)
E_\beta\!\left(-\mu t^\beta |\tilde\lambda|^\alpha\right)
|\tilde\lambda|^{\frac{\kappa_0-1}{2}},
\]
which suggests that the limit field is
\begin{equation*}
\tilde{U}_0^g(t,x)
=
\sqrt{C(\kappa_0,A_0)}
\int_{\mathbb R}
\widehat g(\lambda,x)
E_\beta\!\left(-\mu t^\beta |\lambda|^\alpha\right)
|\lambda|^{\frac{\kappa_0-1}{2}}
W(d\lambda).
\end{equation*}
As $\widehat g(\cdot,x)\in L_2(\mathbb R)$ and, by (\ref{upperch5}), $E_\beta(-z)\in(0,1],$ $z\ge0,$ this field has a finite variance and, therefore, is correctly defined when $\widehat g(\cdot,x)|\cdot|^{\frac{\kappa_0-1}{2}}\in L_2(\mathbb R).$

Let us justify the convergence
\[
R_\varepsilon(t,x)=\mathbb{E}\bigl(U_\varepsilon^g(t,x)-\tilde{U}_0^g(t,x)\bigr)^2\to 0,\quad \mbox{when}\ \varepsilon \to 0.
\]
By the It\^o isometry,
\[
R_\varepsilon(t,x)
=
\int_{\mathbb R}
|\widehat g(\lambda,x)|^2
\left(
E_\beta\!\left(-\mu |\lambda|^\alpha (1+|\lambda|^2\varepsilon^{2\rho_3})^{\gamma/2} t^\beta\right)
\sqrt{\frac{f(\varepsilon^{\rho_3}\lambda)}{\varepsilon^{\rho_3(\kappa_0-1)}}}\right.\]
\[\left.
-\,
\sqrt{C(\kappa_0,A_0)}\,
E_\beta\!\left(-\mu t^\beta |\lambda|^\alpha\right)
|\lambda|^{\frac{\kappa_0-1}{2}}
\right)^2
\,d\lambda.
\]

The integrand converges pointwise to zero. Noting that for sufficiently small $\varepsilon$,
${f(\varepsilon^{\rho_3}\lambda)}/{\varepsilon^{\rho_3(\kappa_0-1)}}
\le C|\lambda|^{\kappa_0-1},$
one can see that in a neighbourhood of the origin, the integrand is dominated by
$
C_1|\widehat g(\lambda,x)|^2|\lambda|^{\kappa_0-1},
$
and therefore has an integrable majorant on $\mathbb R$.
Hence, one can apply the generalised Lebesgue dominated convergence theorem, and the remaining part of the proof uses the same arguments as in the proof of \cite[Theorem 3.2]{Alghamdi2025II}. Similar to Theorem~\ref{ThmP2.1}, in contrast to \cite[Theorem 3.2]{Alghamdi2025II}, the condition $\alpha > \kappa_0/2$ is not required due to the presence of the kernel multiplier.

The covariance representation follows directly from the It\^o isometry, which completes the proof.
\end{proof}

\begin{remark}
Note that, in general, the product $\widehat{g}(\lambda,x)\overline{\widehat{g}(\lambda,x')}$ is not a function of $x - x'$. Consequently, in contrast to the results in \cite{Alghamdi2024}, the limit fields ${U}_0^{g}(t,x)$ and $\tilde{U}_0^{g}(t,x)$ are, in general, non-stationary in both space and time.
\end{remark}

\begin{example} \label{numexP22}

In this example, we considered the limit random field $\tilde{U}_0^g(t,x)$, $t>0$, $x\in\mathbb{R},$ obtained in Theorem~\ref{ThmP2.2} under Assumption~{\rm\ref{Asumptionmain}} with $A_0 \ne 0.$
 The numerical parameters were kept the same as in Example~\ref{numexP21}. In addition, for this case, the values $A_0=0.4$ and $\kappa_0=0.2$ were used. The same kernels as in~\eqref{exampv=1/2} and~\eqref{examplev=3/2}  were employed. The field was evaluated on the space--time grid described in Example~\ref{numexP21}.

The limit field is simulated using the spectral representation in~\eqref{RTH3.2A0>0} and its Riemann-type approximation
\[\tilde{U}_0^g(t,x) \approx
\sqrt{C(\kappa_0,A_0)}
\sum_{j=-N}^{N}
\widehat{g}(\lambda_j,x)\,
E_{\beta}\!\left(-\mu t^{\beta}|\lambda_j|^{\alpha}\right)
|\lambda_j|^{(\kappa_0-1)/2}
\,W(\Delta \lambda_j).
\]

The additional factor $|\lambda_j|^{(\kappa_0-1)/2}$ is the main numerical difference from the example for Theorem~\ref{ThmP2.1}. It gives stronger weight to low frequencies, and therefore, the behaviour of the field is more influenced by the neighbourhood of the origin in the spectral domain.

Figure~\ref{fig:P2th3.2a} shows the realization of
$\tilde{U}_0^g(t,x)$ for $\nu=1/2.$ Similar to Example~\ref{numexP21}, the main variation of the field remains in the spatial direction and it gradually decreases over the temporal domain.
The covariance function in Figure~\ref{fig:P2th3.2b} is largest near the selected spatial reference location and decreases as the spatial or temporal separation increases. The covariance surface is more elevated compared to Figure~\ref{fig:P2th3.1b}.
Figures~\ref{fig:P2th3.2c} and~\ref{fig:P2th3.2d} show the spatial and temporal covariance functions for fixed temporal and spatial variables, respectively. Their shapes are similar to those in Example~\ref{numexP21}, but with higher values. This difference in magnitude is due to the stronger contribution of low-frequency components in the case $A_0>0$.
  \begin{figure}[!ht]
    \centering
    \begin{subfigure}{0.49\textwidth}
        \centering
        \includegraphics[width=\linewidth, height=0.8\linewidth, trim=15mm  10mm 5mm 35mm,clip]{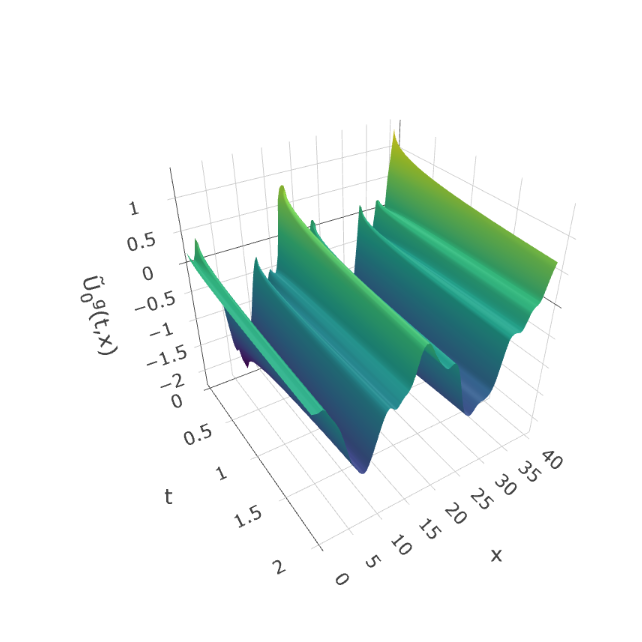}
        \caption{Realization of $\tilde{U}_0^g(t,x)$}
        \label{fig:P2th3.2a}
    \end{subfigure}\hfill
    \begin{subfigure}{0.49\textwidth}
        \centering
        \includegraphics[width=\linewidth, height=0.8\linewidth, trim=15mm  10mm 5mm 35mm,clip]{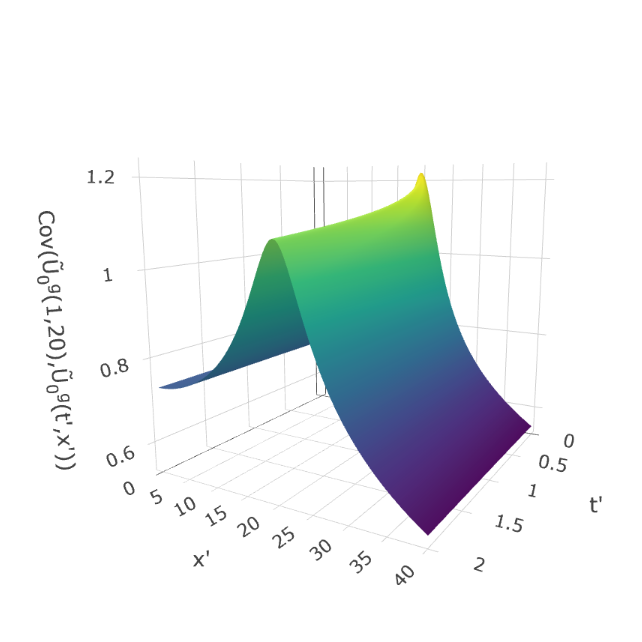}
        \caption{Covariance function of $\tilde{U}_0^g(t,x)$}
        \label{fig:P2th3.2b}
    \end{subfigure}
    \begin{subfigure}{0.48\textwidth}
        \centering
        \includegraphics[width=\linewidth, height=0.8\linewidth]{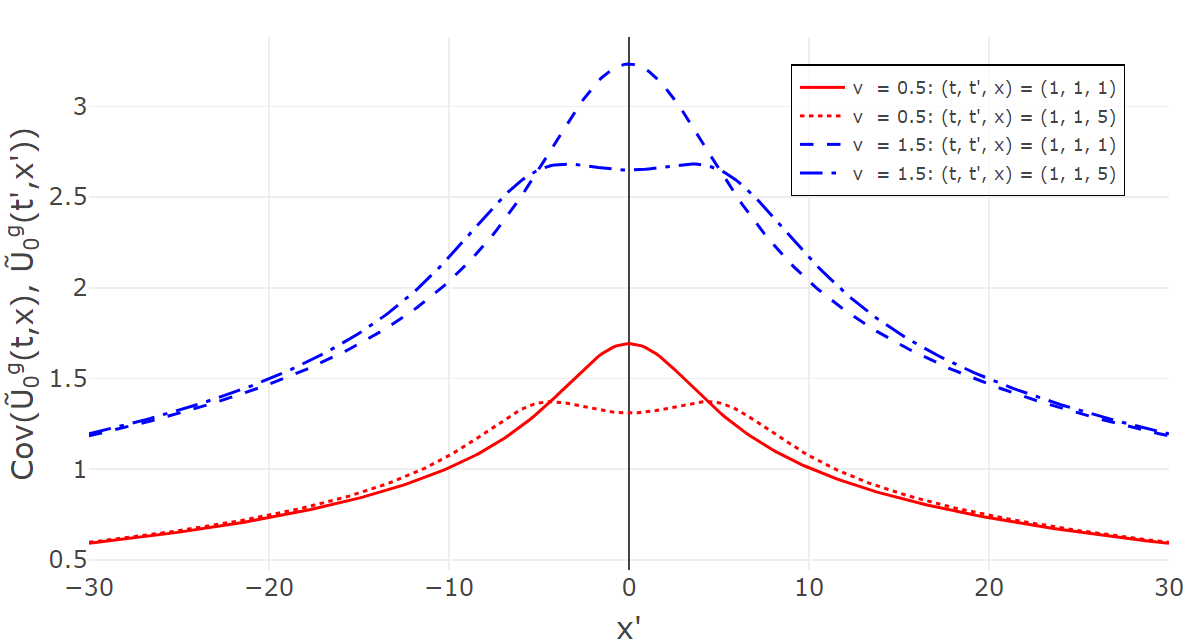}
        \caption{Examples of spatial covariance functions, for fixed $t$, $t'$ and $\nu=0.5$ and $1.5$}
        \label{fig:P2th3.2c}
    \end{subfigure}\hfill
    \begin{subfigure}{0.48\textwidth}
        \centering
        \includegraphics[width=\linewidth, height=0.8\linewidth]{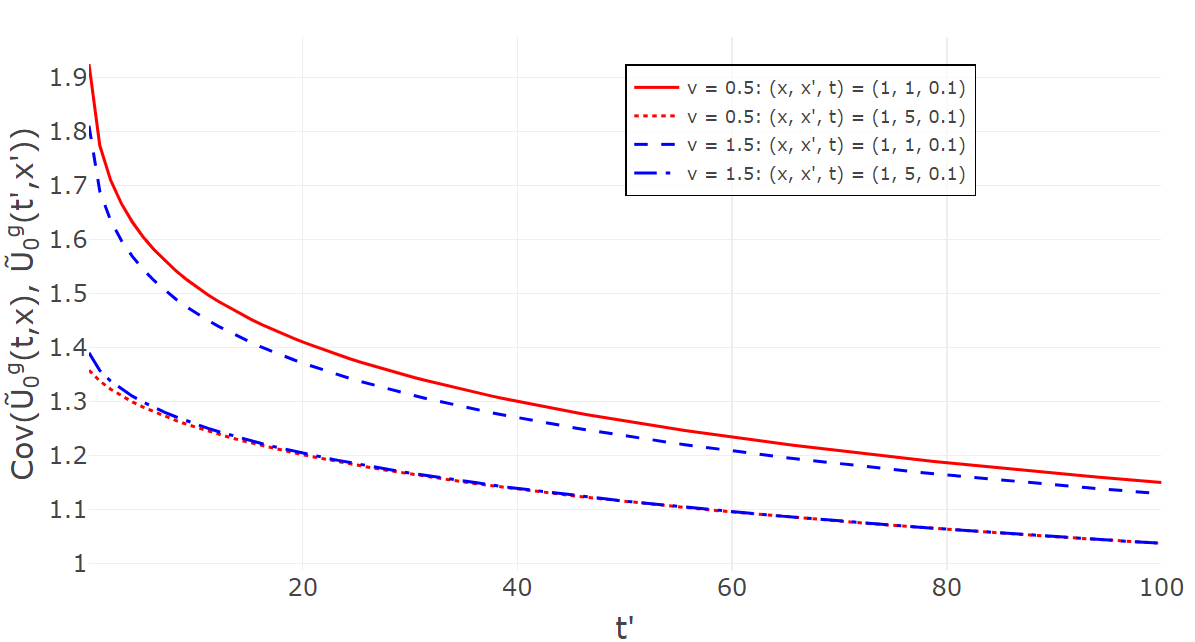}
        \caption{Examples of temporal covariance functions, for fixed $x,$ $x'$ and $\nu=0.5$ and $1.5$}
        \label{fig:P2th3.2d}
    \end{subfigure}
    \caption{Example of the limit field $\tilde{U}_0^g(t,x)$ from Theorem~\ref{ThmP2.2}}
    \label{fig:P2th3.2}
\end{figure}
\end{example}

\section{Some properties of multiscaling limit fields}\label{Holder}
       This section studies properties of the multiscaling limit fields introduced in the previous section, demonstrating how they depend on the smoothing kernels and parameters of the FRBE.

    Let us define  \[
    \eta_*=
    \sup\left\{
    \eta\in(0,1]:
    \int_{\mathbb R}
    |\lambda|^{2\alpha\eta}
    |\widehat g(\lambda,x)|^2
    d\lambda<\infty\ \mbox{for all}\ x\in\mathbb{R}
    \right\}
    \]
   and
     \[
    \tilde\eta_*=
    \sup\left\{
    \eta\in(0,1]:
    \int_{\mathbb R}
    |\lambda|^{2\alpha\eta + \kappa_{0}-1}
    |\widehat g(\lambda,x)|^2
    d\lambda<\infty\ \mbox{for all}\ x\in\mathbb{R}
    \right\}.
    \]
\begin{example}
For the case of the Mat\'ern kernels, using $\widehat g(\lambda,x)$ given by~\eqref{MFourier}, one obtains
\[
\int_{\mathbb R}
|\lambda|^{2\alpha \eta}
|\widehat g(\lambda,x)|^2\,d\lambda
=
C
\int_{\mathbb R}
\frac{|\lambda|^{2\alpha \eta}}
{(a^2+\lambda^2)^{2\nu+1}}
\,d\lambda
< +\infty .
\]
As $\nu>0$ and $\alpha>0$, this integral is finite when
$4\nu+2-2\alpha\eta>1.$ Therefore, $\eta_*=\min\left(1,\frac{4\nu+1}{2\alpha}\right).$ Similarly, one obtains $\tilde{\eta}_*
=
\min\left(1,\frac{4\nu+2-\kappa_0}{2\alpha}\right).$

  \end{example}

   First, we investigate the Hölder continuity of the limit fields in time.
    \begin{thm} \label{ThmP2.3}
    Assume that there exist $\eta_*>0$ for the kernel $g(\cdot,\cdot).$ Then the limit field $U_0^g(t,x)$ in Theorem~\ref{ThmP2.1} is mean-square H\"older continuous in $t$ of any order  $\gamma_t\in (0,\eta_*\beta).$ It has a continuous modification with sample paths that are almost surely H\"older continuous in $t$ of order $\gamma_t.$

If there exists $\tilde \eta_*>0$ for the kernel $g(\cdot,\cdot)$, then the limit field $\tilde U_0^g(t,x)$ in Theorem~\ref{ThmP2.2} is mean-square H\"older continuous in $t$ and has a continuous modification with sample paths that are almost surely H\"older continuous in $t$ of any order $\gamma_t\in (0,\tilde{\eta}_*\beta).$
        \end{thm}
 \begin{proof}
Let $x\in\mathbb R$ be fixed.
    Using the spectral representation of $U_0^g(t,x)$ in Theorem~\ref{ThmP2.1} and the It\^o isometry for stochastic integrals with respect to the Wiener measure, one obtains that for $t,$ $s >0$ it holds that
\begin{align}
    \mathbb E\left|U_0^g(t,x)-U_0^g(s,x)\right|^2
   &=
   C(\bar{\kappa},\bar{w},\bar{A}) \nonumber
\times  \int_{\mathbb R}
    |\widehat g(\lambda,x)|^2
    \left|
    E_\beta\!\left(-\mu t^\beta|\lambda|^\alpha\right) \right. \\
&\left.\quad\quad-
    E_\beta\!\left(-\mu s^\beta|\lambda|^\alpha\right)
    \right|^2
    d\lambda. \label{ito}
\end{align}

  It follows from local smoothness of the Mittag-Leffler function that on bounded intervals it is Lipschitz, and there exists a constant $C>0$ such that for any $\eta\in(0,1]$
    \[
    |E_\beta(-v_1)-E_\beta(-v_2)|
    \le
    C |v_1-v_2|^\eta.
    \]
    Taking $v_1=\mu t^\beta|\lambda|^\alpha$ and $v_2=\mu s^\beta|\lambda|^\alpha,$
   since $\beta \in (0,1]$,    we get
\[
\left|
E_\beta\!\left(-\mu t^\beta|\lambda|^\alpha\right)
-
E_\beta\!\left(-\mu s^\beta|\lambda|^\alpha\right)
\right|
\le
C |\lambda|^{\eta\alpha} |t^\beta-s^\beta|^\eta\le C |\lambda|^{\eta\alpha} |t-s|^{\eta\beta}.
\]

Using the It\^o isometry, we obtain
    \[
    \mathbb E\left|U_0^g(t,x)-U_0^g(s,x)\right|^2
    \le
    C |t-s|^{2\eta\beta}
    \int_{\mathbb R}
    |\lambda|^{2\eta\alpha}
    |\widehat g(\lambda,x)|^2
    d\lambda .
    \]

Thus, as the last integral is finite for $\eta<\eta_*,$  $U_0^g(\cdot,x)$ is mean-square H\"older continuous in $t$ of order
   $
    \eta\beta.$

Moreover, since $U_0^g(t,x)$ is Gaussian, for every $p\ge 2$,
    \[
    \mathbb E\left|U_0^g(t,x)-U_0^g(s,x)\right|^p
    =
    C_p
    \left(
    \mathbb E\left|U_0^g(t,x)-U_0^g(s,x)\right|^2
    \right)^{p/2}
    \le
    C |t-s|^{\eta p\beta}.
    \]
    By Kolmogorov's continuity theorem, the field has a continuous modification with sample paths that are almost surely H\"older
    continuous in $t$ of any order
   $
    \gamma_t\in (0,\eta_*\beta-1/p).$ As~$p$ can be selected arbitrary large, we obtain that the  H\"older exponent in $t$ can be chosen as $\gamma_t\in (0,\eta_*\beta).$

    For the limit field $\tilde U_0^g(t,x)$ in Theorem~\ref{ThmP2.2} and $\gamma_t\in (0,\tilde\eta_*\beta)$ the proof is identical, noting that
\begin{align*}
\mathbb E\left|\tilde U_0^g(t,x)-\tilde U_0^g(s,x)\right|^2
& =
C(\kappa_0,A_0)
\int_{\mathbb R}
|\widehat g(\lambda,x)|^2
\left|
E_\beta\!\left(-\mu t^\beta|\lambda|^\alpha\right)\right.
\\
&\left.\quad\quad -
E_\beta\!\left(-\mu s^\beta|\lambda|^\alpha\right)
\right|^2
|\lambda|^{\kappa_0-1}
\,d\lambda
\end{align*}
and
\[
\mathbb E\left|\tilde U_0^g(t,x)-\tilde U_0^g(s,x)\right|^2
\le
C |t-s|^{2\eta\beta}
\int_{\mathbb R}
|\lambda|^{2\eta\alpha+\kappa_0-1}
|\widehat g(\lambda,x)|^2
d\lambda .\qedhere
\]
\end{proof}
The next result examines the Hölder continuity of the limit field with respect to the spatial variable.
  \begin{thm} \label{ThmP2.4}
  Assume that the kernel function $g(\cdot,\cdot)$ has a finite Sobolev norm of order~$-\alpha$ with respect to the first variable, and assume that this norm is H\"older continuous of order $\theta \in (0,1]$ with respect to the second variable, that is,
      \[
      \left\|
      g(z,x)-g(z,y)
      \right\|_{H_z^{-\alpha}}
      \le
      C |x-y|^\theta.
      \] Then, the limit random field $U_0^g(t,x)$ in Theorem~\ref{ThmP2.1} is mean-square H\"older continuous in $x$ of any order  $\gamma_x\in (0,\theta).$ It has a continuous modification with sample paths that are almost surely H\"older continuous in $x$ of order $\gamma_x.$

If the following assumption holds true
\[
\left\|
(-\Delta)_z^{(\kappa_0-1)/4}(g(z,x)-g(z,y))
\right\|_{H_z^{-\alpha}}
\le
C |x-y|^{\theta^*},
\]
then, the limit random field $\tilde U_0^g(t,x)$ in Theorem~\ref{ThmP2.2} is mean-square H\"older continuous, and it admits a modification with almost surely continuous sample paths in $x$ of the H\"older order $\gamma_x\in (0,\theta^*)$.
\end{thm}
\begin{proof}
   Let $t>0$ be fixed. By the It\^o isometry, we have
  \[
  \begin{aligned}
      \mathbb{E}\left|U_0^g(t,x)-U_0^g(t,y)\right|^2
      &=
      \int_{\mathbb{R}}
      \left|
      \widehat{g}(\lambda,x)-\widehat{g}(\lambda,y)
      \right|^2
      \left|
      E_\beta\!\left(-\mu t^\beta |\lambda|^\alpha\right)
      \right|^2
      \,d\lambda .
  \end{aligned}
  \]

  Using the estimate (\ref{upperch5}) one obtains that there exists a such constant $C>0$ that
  \[
  \left|
  E_\beta\!\left(-\mu t^\beta |\lambda|^\alpha\right)
  \right|^2
  \le
  \left(1 + \frac{\mu t^\beta |\lambda|^\alpha}{\Gamma(1+\beta)}\right)^{-2} \le
  \frac{C}{\left(1+  |\lambda|^\alpha\right)^2}\le
  \frac{C}{\left(1+  |\lambda|^2\right)^\alpha}.
  \]

  Hence,
  \[
  \begin{aligned}
      \mathbb{E}\left|U_0^g(t,x)-U_0^g(t,y)\right|^2
      &\le
      C
      \int_{\mathbb{R}}
      \frac{
          \left|
           \widehat{g}(\lambda,x)-\widehat{g}(\lambda,y)
          \right|^2
      }{
          \left(1+ |\lambda|^2\right)^\alpha
      }
      \,d\lambda .
  \end{aligned}
  \]
The last integral is the squared negative Sobolev norm with respect to the first variable of the function $g(\cdot,x)-g(\cdot,y).$
 Therefore, by the conditions of the theorem,
  \[
      \mathbb{E}\left|U_0^g(t,x)-U_0^g(t,y)\right|^2
      \le
      C
      \left\|
      g(z,x)-g(z,y)
      \right\|_{H_z^{-\alpha}}^2\le C |x-y|^{2\theta}.
  \]
  Thus, $U_0^g(t,x)$ is mean-square H\"older continuous in $x$ of any order $\gamma_x \in (0,\theta)$. By Kolmogorov's continuity theorem, it admits a modification whose sample paths are H\"older continuous of order $\gamma_x$.

In the case of the limit field $\tilde U_0^g(t,x)$ in Theorem~\ref{ThmP2.2} we analogously obtain that
\begin{equation}\label{upperbp}
\mathbb{E}\left|\tilde U_0^g(t,x)-\tilde U_0^g(t,y)\right|^2
\le
C
\int_{\mathbb{R}}
\frac{
\left|
\widehat{g}(\lambda,x)-\widehat{g}(\lambda,y)
\right|^2
}{
\left(1+ |\lambda|^2\right)^\alpha
}
|\lambda|^{\kappa_0-1}
\,d\lambda .
\end{equation}

By the properties of the Fourier transforms of the fractional Laplacian
\[
\widehat{(-\Delta)_z^{c/2} g(z,x)}
=
|z|^c\, \widehat{g}(z,x),
\qquad c\in(0,2),
\]
one obtains that the last integral in (\ref{upperbp}) is equal to \[
\left\|
(-\Delta)_z^{(\kappa_0-1)/4}(g(z,x)-g(z,y))
\right\|^2_{H_z^{-\alpha}}.
\] Thus, the assumption of the theorem implies the required statement. \end{proof}
    \begin{example}
For the Mat\'ern kernel case, it follows by the formula (\ref{MFourier}) that    \[
I(x,y):=\int_{\mathbb{R}}
\frac{
    \left|
    \widehat{g}(\lambda,x)-\widehat{g}(\lambda,y)
    \right|^2
}{
    \left(1+ |\lambda|^2\right)^\alpha
}
\,d\lambda
    =C
    \int_{\mathbb{R}}
    \frac{
        \left|
        e^{i\lambda x}-e^{i\lambda y}
        \right|^2 \,d\lambda
    }{
        \left(1+|\lambda|^2\right)^{\alpha
    }\left( {a^2} + \lambda^2 \right)^{2\nu + 1}}\]
    \[
   \le C
    \int_{\mathbb{R}}
    \frac{
        \left|
        e^{i\lambda x}-e^{i\lambda y}
        \right|^2 \,d\lambda
    }{
        \left(1+|\lambda|^2\right)^{\alpha+2\nu + 1
        }}.
    \]

Using the elementary estimate
  $
    \left|e^{i\lambda x}-e^{i\lambda y}\right|
    \le C\min\{|\lambda||x-y|,1\}
$
one obtains
\begin{equation}\label{upperI}
    I(x,y)
    \le
    C\int_{\mathbb R}
    \frac{
        \min(\lambda^2|x-y|^2,1)
    }{
        (1+\lambda^2)^{\alpha+2\nu + 1}
    }
    \,d\lambda .
\end{equation}

Consider $ 0<|x-y|\le 1$ and split the integral in (\ref{upperI}) at
$|\lambda|=|x-y|^{-1}.$
Then

    \[
    I(x,y)
    \le
    C|x-y|^2
    \int_{|\lambda|\le |x-y|^{-1}}
    \frac{\lambda^2\,d\lambda}{(1+\lambda^2)^{\alpha+2\nu + 1}}
       +
    C
    \int_{|\lambda|>|x-y|^{-1}}
    \frac{d\lambda}{(1+\lambda^2)^{\alpha+2\nu + 1}} .
    \]

Noting that $\alpha$ and $\nu>0,$ then  $\alpha+2\nu + 1>1/2$, the second integral is finite and
    \[
    \int_{|\lambda|>|x-y|^{-1}}
    \frac{d\lambda}{(1+\lambda^2)^{\alpha+2\nu + 1}}
    \leq C
    \int_{|\lambda|>|x-y|^{-1}}
    |\lambda|^{-2({\alpha+2\nu + 1})}
    \,d\lambda
    \leq C
    |x-y|^{2{\alpha+4\nu + 1}}.
    \]

   Since for large $|\lambda|$  it holds that
    \[
    \frac{\lambda^2}{(1+\lambda^2)^{\alpha+2\nu + 1}}
    \leq C |\lambda|^{-2{\alpha-4\nu}},
    \]
   one obtains  for the first integral
    \[
   |x-y|^2 \int_{|\lambda|\le |x-y|^{-1}}
    \frac{\lambda^2}{(1+\lambda^2)^{\alpha+2\nu + 1}}
    \,d\lambda
    \leq C
    \begin{cases}
        |x-y|^{2\alpha+4\nu+1}, & 2\alpha+4\nu<1,\\[1mm]
        -|x-y|^2\log(|x-y|), & 2\alpha+4\nu=1,\\[1mm]
        |x-y|^2, & 2\alpha+4\nu>1.
    \end{cases}
    \]

   Thus,  by combining the two estimates,
    \[
    I(x,y)\leq C
     \begin{cases}
        |x-y|^{2\alpha+4\nu+1}, & 2\alpha+4\nu<1,\\[1mm]
        -|x-y|^2\log(|x-y|), & 2\alpha+4\nu=1,\\[1mm]
        |x-y|^2, & 2\alpha+4\nu>1.
    \end{cases}
    \]

    On the other hand, if $|x-y|\ge 1$, then, for any $\theta>0,$ it holds
      \[
    I(x,y)
    \le
    C \int_{\mathbb{R}}
    \frac{d\lambda
    }{
        \left(1+|\lambda|^2\right)^{\alpha+2\nu + 1
    }} \le
C|x-y|^{2\theta}.
    \]

Consequently, in the Mat\'ern kernel case, Theorem~\ref{ThmP2.4} holds for the limit random field $U_0^g(t,x)$ provided that
    $\theta=\min\left(1,\alpha+2\nu+1/2\right).$
        Analogously, for the limit random field $\tilde{U}_0^g(t,x)$, it is required that $
        \theta^*=\min\left(1,\alpha+2\nu+1-{\kappa_0}/{2}\right).
     $
          \end{example}

Now we investigate the dependence structure of the limit fields. Recall that, in general, these random fields are nonstationary in both space and time. Therefore, we characterise their short- and long-range dependence separately for each variable and at each space--time location.

We say that a random field $\zeta(x,t)$ is short-range dependent in time at a point $(x_0, t_0),$ $x_0\in\mathbb{R},$ $t_0>0,$ if its covariance function is absolutely integrable with respect to time lag
\begin{equation}\label{SHORT_t}
\int_0^{+\infty}|\operatorname{Cov}\bigl(\zeta(t_0,x_0),\zeta(t_0+h,x_0)\bigr)|dh<+\infty.
\end{equation}
It is  short-range dependent in space at $(x_0,t_0)$ if
\begin{equation}\label{SHORT_x}
\int_{\mathbb R}|\operatorname{Cov}\bigl(\zeta(t_0,x_0),\zeta(t_0,x_0+h)\bigr)|dh<+\infty.
\end{equation}
If the above integrals are divergent, the fields are called long-range dependent.

\begin{thm}  \label{ThmP2.6}
Let $\beta\in(0,1)$ and the kernel function $g(x_1,x)$ have a nondegenerate Fourier transform $\hat{g}(\lambda,x_0)$, i.e., $\hat{g}(\lambda,x_0) \neq 0$ on a set of positive measure. Then, the limit fields $U_0^{g}(t,x)$ and $\tilde{U}_0^{g}(t,x)$ are long-range dependent in time at $(x_0,t_0)$.

If for almost all $\lambda\in\mathbb R$ it holds that uniformly $\hat{g}(\lambda,\cdot)\in L_1(\mathbb R)$, then the limit fields $U_0^{g}(t,x)$ and $\tilde{U}_0^{g}(t,x)$ are  short-range dependent in space at $(x_0,t_0)$.
\end{thm}
\begin{proof}
Note that, by Theorem~\ref{ThmP2.1} and the positivity of the Mittag-Leffler functions for negative arguments, it holds
\begin{align}
&\int_0^{+\infty}\left|\operatorname{Cov}\bigl(U_0^g(t_0,x_0),U_0^g(t_0+h,x_0)\bigr)\right|dh\nonumber\\
&\quad =
C(\bar{\kappa},\bar{w},\bar{A})
 \int_0^{+\infty}\left|
\int_{\mathbb{R}}
|\widehat{g}(\lambda,x_0)|^2
E_\beta(-\mu t_0^\beta |\lambda|^\alpha)
E_\beta(-\mu (t_0+h)^\beta |\lambda|^\alpha)\,d\lambda\right|dh
\nonumber\\
&\quad =
C(\bar{\kappa},\bar{w},\bar{A})
\int_{\mathbb{R}}
|\widehat{g}(\lambda,x_0)|^2
E_\beta(-\mu t_0^\beta |\lambda|^\alpha)
\int_0^{+\infty}E_\beta(-\mu (t_0+h)^\beta |\lambda|^\alpha)dh\,d\lambda.\label{covint}
\end{align}
It follows from (\ref{upperch5}) and $\beta\in(0,1)$ that
\[\int_0^{+\infty}E_\beta(-\mu (t_0+h)^\beta |\lambda|^\alpha)dh\ge \int_0^{+\infty}\frac{dh}{1 + \Gamma(1-\beta)\mu|\lambda|^\alpha (t_0+h)^\beta }=+\infty.\]
Thus, the integral in (\ref{covint}) is divergent and the limit random field is long-range dependent in time at $(x_0,t_0)$.

By Theorem~\ref{ThmP2.1}, for the spatial covariance function
\begin{align*}
&\int_{\mathbb R}\left|\operatorname{Cov}\bigl(U_0^g(t_0,x_0),U_0^g(t_0,x_0+h)\bigr)\right|dh\\
&\quad \le
C(\bar{\kappa},\bar{w},\bar{A})
\times \int_{\mathbb R}
\int_{\mathbb R}
|\widehat{g}(\lambda,x_0)||\widehat{g}(\lambda,x_0+h)|
E_\beta^2(-\mu t_0^\beta |\lambda|^\alpha)
\,d\lambda dh\\
&\quad  =C(\bar{\kappa},\bar{w},\bar{A})
\times
\int_{\mathbb R}
|\widehat{g}(\lambda,x_0)|\int_{\mathbb R}|\widehat{g}(\lambda,h)|dh\,
E_\beta^2(-\mu t_0^\beta |\lambda|^\alpha)
\,d\lambda<+\infty,
\end{align*}
which implies spatial short-range dependence.
The proof for the limit field $\tilde{U}_0^{g}(t,x)$ follows by the same arguments.
\end{proof}

\begin{remark}
By the Cauchy--Schwarz inequality,
\[\|\hat{g}(\lambda,\cdot)\|_{L_1(\mathbb{R})}
\le
\left(\int_{\mathbb{R}} (1+|x|^2)^{-s}\, dx \right)^{1/2}
\left(\int_{\mathbb{R}} (1+|x|^2)^s |\hat{g}(\lambda,x)|^2\, dx \right)^{1/2}.
\]
Therefore, if
$g(t_0,\cdot) \in H^s(\mathbb{R})$ for some $s > {1}/{2}$,
then the Fourier transform $\hat{g}(\lambda,\cdot) \in L_1(\mathbb{R}),$
which guaranties short-range dependence in space of the limit random fields $U_0^{g}(t,x)$ and $\tilde{U}_0^{g}(t,x)$ at $(x_0,t_0)$.
\end{remark}
 \begin{example}
For the case of Mat\'ern kernels, the Fourier transform $\widehat g(\lambda,x)$ is given by~\eqref{MFourier} and is nondegenerate. Therefore, the corresponding limit fields $U_0^g(t,x)$ and $\widetilde U_0^g(t,x)$ are long-range dependent in time for all $(x_0,t_0).$

Notice that by~\eqref{MFourier} the function $\widehat g(\lambda,\cdot)\notin L_1(\mathbb R).$ Hence, the second part of Theorem~\ref{ThmP2.6} can't be applied directly in the Mat\'ern case. However, in special cases, one can explicitly compute the covariance function and check condition~\eqref{SHORT_x}.

Let $\alpha=\beta=a=1.$ Then, as $E_1(z)=e^z,$ by \eqref{CovTH3.1a0=0} and \eqref{MFourier},
\[
\operatorname{Cov}\bigl(U_0^g(t_0,x_0),U_0^g(t_0,x_0+h)\bigr)
=
C\int_{\mathbb R}
e^{-i\lambda h}
e^{-2\mu t_0|\lambda|}
(1+\lambda^2)^{-(2\nu+1)}
\,d\lambda .
\]

First, note that since
\[
P_{\mu,t_0}(h)
:=
\int_{\mathbb R}
e^{-i\lambda h}
e^{-2\mu t_0|\lambda|}
\,d\lambda
=
\frac{4\mu t_0}{(2\mu t_0)^2+h^2},
\]
then $P_{\mu,t_0}(\cdot)\in L_1(\mathbb R).$

Next, by computing the following integral we obtain
\[
G_\nu(h)
:=
\int_{\mathbb R}
e^{-i\lambda h}
(1+\lambda^2)^{-(2\nu+1)}
\,d\lambda
=
\frac{2\sqrt{\pi}}{\Gamma(2\nu+1)}
\left(\frac{|h|}{2}\right)^{2\nu+\frac12}
K_{2\nu+\frac12}(|h|).
\]

Since $ K_{2\nu+\frac12}(|h|)\sim C|h|^{-2\nu-\frac12},$ when
$ |h|\to 0,$ \cite[formula 9.6.9]{abramowitz1972}, and \\ $K_{2\nu+\frac12}(|h|)\sim C|h|^{-1/2}e^{-|h|},$ when $
|h|\to\infty,$ \cite[formula 9.7.2]{abramowitz1972}, it follows that $
G_\nu(\cdot)\in L_1(\mathbb R).$

Because the covariance function $\operatorname{Cov}\bigl(U_0^g(t_0,x_0),U_0^g(t_0,x_0+h)\bigr)$
is the Fourier transforms of the product $e^{-2\mu t_0|\lambda|}
(1+\lambda^2)^{-(2\nu+1)}$, it can be written as the convolution
\[\operatorname{Cov}\left(U_0^g(t_0,x_0),U_0^g(t_0,x_0+h)\right)
=
C\cdot
(P_{\mu,t_0}*G_\nu)(h).
\]

Hence, by Young's convolution inequality, we obtain
\[
\int_{\mathbb R}\left|\operatorname{Cov}\bigl(U_0^g(t_0,x_0),U_0^g(t_0,x_0+h)\bigr)\right|dh
\le
C
\|P_{\mu,t_0}\|_{L_1}
\|G_\nu\|_{L_1}
<+\infty .
\]
and the limit fields $U_0^{g}(t,x)$ is short-range dependent in space at $(x_0,t_0)$.
  \end{example}

\section{Conclusion} \label{conclusionP2}

This paper investigated the limit behaviour of solutions to the FRBE with random initial conditions exhibiting both classical and cyclic long-memory structures. The multiscaling limit analysis developed in~\cite{Alghamdi2024, Alghamdi2025II} was extended to a broad class of kernel-smoothing transformations. Properties of the limit fields were studied. In particular, the limit fields are non-stationary in both space and time. Their mean-square and sample-path H\"{o}lder continuities in temporal and spatial variables were established. The short/long-range dependence properties of the limit fields in space and time were also analysed. It was demonstrated how the limit fields and their properties depend on the smoothing kernels, the parameters of the FRBE, and the locations of the spectral singularities.

Several directions for future research that arise from the present work include
\begin{enumerate}
\item Verifying the conjecture that the limiting fields are non-degenerate only for the values of the scaling parameters $\rho_1$, $\rho_2$, and $\rho_3$ specified in this paper;
\item Extending the asymptotic framework to settings with regularly varying spectral densities (see, for example, \cite{Leonenko2013II, Olenko2005, Olenko2007});
\item Generalising the results to subordinated initial conditions (see, for example, \cite{Anh1999II, Leonenko1999, Knopova2004});
\item Extending the analysis to multidimensional spatio-temporal random fields (see \cite{Anh2000, Anh2002, Olenko2013II});
\item Applying the current methodology to other classes of stochastic partial differential equations (for example, \cite{Angulo2000, kozachenko2020}).
\end{enumerate}

\begin{funding}
This research was supported by the Australian Research Council's Discovery Projects funding scheme (project number DP220101680).  A.~Olenko was also partially supported by La Trobe University's SCEMS CaRE and Beyond grant.
\end{funding}

\bibliographystyle{vmsta2-mathphys}
\bibliography{Bibliography}

@article{Anh2003,
title = {Harmonic analysis of random fractional diffusion–wave equations},
journal = {Applied Mathematics and Computation},
volume = {141},
number = {1},
pages = {77 -- 85},
year = {2003},
author = {Anh, V. and  Leonenko, N.}, 
doi={https://doi.org/10.1016/S0096-3003(02)00322-3}
}

@book{Gorenflo2020,
  title={{Mittag-Leffler} Functions, Related Topics and Applications},
  author={Gorenflo, R. and Kilbas, A. A. and Mainardi, F. and Rogosin, S. V. and others},
  year={2014},
  publisher={Springer},   address = {Berlin}, doi={10.1007/978-3-662-43930-2}
}

@article{Anh2000,
  title   = {Scaling laws for fractional diffusion--wave equations with singular data},  author  = {Anh, V. and Leonenko, N.},
  journal = {Statistics \& Probability Letters},
  year    = {2000},
  volume  = {48},
  number  = {3},
  pages   = {239 -- 252},
  doi     = {10.1016/S0167-7152(00)00003-1}
}

@article{Alghamdi2024,
  author  = {Alghamdi, M. M. A. and Leonenko, N. and Olenko, A.},
  title   = {Multiscaling limit theorems for stochastic {FPDE} with cyclic long-range dependence},
  journal = {Brazilian Journal of Probability and Statistics},
  year    = {2025},
  volume  = {39},
  number  = {2},
  pages   = {226--247},
  doi     = {10.1214/25-BJPS633}
}

@article{Caputo1967,
  author  = {Caputo, M.},
  title   = {Linear model of dissipation whose {\(Q\)} is almost frequency independent, {II}},
  journal = {Geophysical Journal International},
  volume  = {13},
  number  = {5},
  pages   = {529 -- 539},
  year    = {1967}, doi={10.1111/j.1365-246X.1967.tb02303.x}
}

@article{Metzler1994,
  author  = {Metzler, R. and Gl\"{o}ckle, W. G. and Nonnenmacher, T. F.},
  title   = {Fractional model equation for anomalous diffusion},
  journal = {Physica A},
  volume  = {211},
  pages   = {13 -- 24}, year  = {1994}, doi={10.1016/0378-4371(94)90064-7}
}

@book{Podlubny1998,
  title     = {Fractional Differential Equations: An Introduction to Fractional Derivatives, Fractional Differential Equations, to Methods of Their Solution and Some of Their Applications},
  author    = {Podlubny, I.},
  year      = {1998},
  publisher = {Academic Press},
  address   = {San Diego}, doi={10.1016/S0076-5392(99)80031-9 }
}

@book{Oldham1974,
  author    = {K. B. Oldham and J. Spanier},
  title     = {The Fractional Calculus: Theory and Applications of Differentiation and Integration to Arbitrary Order},
  publisher = {Academic Press},
  address   = {New York},
  year      = {1974}, doi={10.1016/s0076-5392(09)x6012-1.}
}

@article{Gay1990,
  title={On a class of random field models which allows long range dependence},
  author={Gay, R. and Heyde, C. C.},
  journal={Biometrika},
  volume={77},
  number={2},
  pages={401 -- 403},
  year={1990}, doi={10.1093/biomet/77.2.401}
}

@article{Angulo2000, 
title={Fractional diffusion and fractional heat equation},
volume={32},
number={4}, 
journal={Advances in Applied Probability}, 
author={J. Angulo and M. D. Ruiz-Medina and V. Anh and W. Grecksch}, year={2000},
pages={1077 -- 1099}, doi={10.1017/S0001867800010478}
}

@article{Anh1999,
  author    = {V. Anh and J. M. Angulo and M. D. Ruiz-Medina},
  title     = {Possible long-range dependence in fractional random fields},
  journal   = {Journal of Statistical Planning and Inference},
  volume    = {80},
  number    = {1-2},
  pages     = {95 -- 110},
  year      = {1999}, doi={10.1016/S0378-3758(98)00244-4}
}

@article{Albeverio1994,
  author    = {S. Albeverio and S. A. Molchanov and D. Surgailis},
  title     = {Stratified structure of the {Universe} and {Burgers’} equation: A probabilistic approach},
  journal   = {Probability Theory and Related Fields},
  volume    = {100},number={4},
  pages     = {457 -- 484},
  year      = {1994}, doi={https://doi.org/10.1007/BF01268990}
}

@article{Leonenko1998, 
author={Leonenko, N. and  Woyczynski, W.},
  title={Scaling limits of solutions of the heat equation for singular non-{Gaussian} data},
  journal={Journal of Statistical Physics},
  year={1998},
  volume={91}, number={1},
  pages={423 -- 438}, doi={10.1023/A:1023060625577}
  }

@article{Beghin2000,
  title={Gaussian limiting behavior of the rescaled solution to the linear {Korteweg--de Vries} equation with random initial conditions},
  author={L. Beghin and V. Knopova and N. Leonenko and E. Orsingher},
  journal={Journal of Statistical Physics},
  volume={99},
  pages={769 -- 781},number={3},
  year={2000}, doi={10.1023/A:1018687327580}}

@article{Anh2002,
  title={Renormalization and homogenization of fractional diffusion equations with random data},
  author={Anh, V. and  Leonenko, N.},
  journal={Probability Theory and Related Fields},
  year={2002},
  volume={124}, number={3},
  pages={381 -- 408}, doi={10.1007/s004400200217}
  }

@article{Leonenko1998a,
  title={Exact parabolic asymptotics for singular $n$-{D} {Burgers'} random fields: {Gaussian} approximation},
  author={Leonenko, N. and  Woyczynski, W.},
  journal={Stochastic Processes and their Applications},
  year={1998},
  volume={76},number={2},
  pages={141 -- 165},
doi={10.1016/S0304-4149(98)00031-3}
}

@book{Leonenko1999,
  title={Limit Theorems for Random Fields with Singular Spectrum},
  author={N. Leonenko},
  year={1999},
  publisher={Kluwer Academic}, address= {Dordrecht},
  doi={10.1007/978-94-011-4607-4}
}

@article{Broadbrige2020,
  author = {Broadbridge, P. and Kolesnik, A. and Leonenko, N. and Olenko, A. and Omari, D.},
  title = {Spherically restricted random hyperbolic diffusion},
  journal = {Entropy},
  volume = {22},
  number = {2},
  pages={217},
  year = {2020},
doi={10.3390/e22020217}
}

@article{Leonenko2024,
  title={On fractional spherically restricted hyperbolic diffusion random field},
  author={N. Leonenko and A. Olenko and J. Vaz},
  journal={Communications in Nonlinear Science and Numerical Simulation},
  volume={131},
  pages={107866},
  year={2024}, doi={10.1016/j.cnsns.2024.107866}
}

@article{Anh1999II,
  title={Non-{Gaussian} scenarios for the heat equation with singular initial conditions},
  author={Anh, V. and  Leonenko, N.},
  journal={Stochastic Processes and their Applications},
  year={1999},number={1},
  volume={84},
  pages={91 -- 114},
doi={10.1016/S0304-4149(99)00053-8}
}

@article{Mainardi2014,
  title={On some properties of the {Mittag-{Leffler}} function ${E_{\alpha}} (-t^{\alpha})$, completely monotone for $t> 0$ with $0<\alpha< 1$},
  author={Mainardi, F.},
  journal={Discrete and Continuous Dynamical Systems - Series B},
  volume={19},
  number={7},pages = {2267 -- 2278},
  year={2014},
doi={10.3934/dcdsb.2014.19.2267}
}

@article{Anh2021,
  title={Fractional stochastic partial differential equation for random tangent fields on the sphere},
  author={V. Anh and A. Olenko and Y. Wang},
  journal={Theory of Probability and Mathematical Statistics},
  volume={104},
  pages={3 -- 22},
  year={2021},
  doi={https://doi.org/10.1090/tpms/1142}
}

@article{Broadbridge2024,
  author    = {P. Broadbridge and I. Donhauzer and A. Olenko},
  title     = {Stochastic diffusion within expanding space--time},
  journal   = {Zeitschrift f{\"u}r angewandte Mathematik und Physik},
  year      = {2024},
  volume    = {75},
  number    = {2},
  pages     = {42},   
doi     = {10.1007/s00033-024-02191-1}
}

@article{Alghamdi2025,
      title={Multiscaling asymptotic behavior of solutions to random high-order heat equations}, 
      author={Alghamdi, M. M. A. and Leonenko, N. and Olenko, A.},
      year={2025},journal={ Submitted; arXiv:2510.14153},
  pages     = {28}, doi    = {10.48550/arXiv.2510.14153}     
}

@article{Leonenko2013II,
  author  = {Leonenko, N. and Olenko, A.},
  title   = {{Tauberian and Abelian theorems for long-range dependent random fields}},
  journal = {Methodology and Computing in Applied Probability},
  volume  = {15},
  pages   = {715 -- 742},number={4},
  year    = {2013},
  doi={10.1007/s11009-012-9276-9}
}

@article{Olenko2005,
  author       = {Olenko, A.},
  title        = {Tauberian theorems for random fields with an {OR} spectrum. {I}},
  journal      = {Theory of Probability and Mathematical Statistics},
  volume       = {73},
  pages        = {135 -- 149},
  year         = {2005},
  doi={10.1090/S0094-9000-07-00688-6}
}

@article{Olenko2007,
  author       = {Olenko, A.},
  title        = {Tauberian theorem for fields with an {OR} spectrum. {II}},
  journal      = {Theory of Probability and Mathematical Statistics},
  volume       = {74},
  pages        = {93 -- 111},
  year         = {2007},
   doi={https://doi.org/10.1090/S0094-9000-07-00700-4}}

@article{Liu2018,
  author  = {Liu, G-R. and Shieh, N-R.},
  title   = {Multi-scaling limits for time-fractional relativistic diffusion equations with random initial data},
  journal = {Theory of Probability and Mathematical Statistics},
  year    = {2018},
  volume  = {95},
  pages   = {109 -- 130},
  doi={10.1090/tpms/1025}
}

@article{Liu2015,
  title={Multi-scaling limits for relativistic diffusion equations with random initial data},
  author={Liu, G-R. and Shieh, N-R.},
  journal={Transactions of the American Mathematical Society},
  volume={367},
  number={5},
  pages={3423 -- 3446},
  year={2015},
  doi={10.1090/S0002-9947-2014-06498-2}
}

@article{Liu2010,
  title={Scaling limits for some {PDE} systems with random initial conditions},
  author={Liu, G-R. and Shieh, N-R.},
  journal={Stochastic Analysis and Applications},
  volume={28},
  number={3},
  pages={505 -- 522},
  year={2010},
  publisher={Taylor \& Francis},
  doi={10.1080/07362991003704969}
}

@article{Anh2001a,
  title={Spectral Analysis of Fractional Kinetic Equations with Random Data},
  author={Anh, V. and  Leonenko, N.},
  journal={Journal of Statistical Physics},
  year={2001},
  volume={104},
  pages={1349-1387},
  doi={10.1023/A:1010474332598}
  }

@article{Alghamdi2025II,
  title={On asymptotic behavior of solutions to random fractional {Riesz-Bessel} equations with cyclic long memory initial conditions},
  author={Alghamdi, M. M. A. and  Olenko, A.},
  year={2026},
  journal = {to appear in Theory of Probability and Mathematical Statistics},
  doi = {10.48550/arXiv.2512.09308
}
}

@article{Porcu2024,
author = {Porcu, E. and Bevilacqua, M. and  Schaback, R. and Oates, C.J.},
title = {{The Matérn} Model: A Journey Through Statistics, Numerical Analysis and Machine Learning},
volume = {39},
journal = {Statistical Science},
number = {3},
publisher = {Institute of Mathematical Statistics},
pages = {469 -- 492},
year = {2024},
doi = {10.1214/24-STS923},
}

@article{Leonenko2022,
  author    = {Leonenko, N. and Malyarenko, A. and Olenko, A.},
  title     = {On spectral theory of random fields in the ball},
  journal   = {Theory of Probability and Mathematical Statistics},
  year      = {2022},
  volume    = {107},
  pages     = {61--76},
  doi       = {https://doi.org/10.1090/tpms/1175}
}

@article{Zhang2022,
  title={A PDE-based adaptive kernel method for solving optimal filtering problems},
  author={Zhang, Z. and Archibald, R. and Bao, F.},
  journal={Journal of Machine Learning for Modeling and Computing},
  volume={3},
  number={3},
  year={2022},
pages={37--59}, doi = {10.1615/JMachLearnModelComput.2022043526}
}

@incollection{De1956,
  title={Random Solutions of Partial Differential Equations},
  author={De F{\'e}riet, J. K.},
  booktitle={Proceedings of the Third Berkeley Symposium on Mathematical Statistics and Probability},
  volume={3},
  pages={ 199 -– 208},
  year={1956},
publisher={University of California Press}, 
address={ Berkeley},doi={10.1525/9780520350694-013}
}

@article{Rosenblatt1968,
  title={Remarks on the {Burgers} equation},
  author={Rosenblatt, M.},
  journal={Journal of Mathematical Physics},
  volume={9},
  number={7},
  pages={1129 -- 1136},
  year={1968},
  doi={10.1063/1.1664687}}

@incollection{Becus1980,
  title={Variational Formulation of Some Problems for the Random Heat Equation},
editor = {G. Adomian},
  author={B{\'e}cus, G.},
  booktitle={Applied Stochastic Processes},
  pages={19 -- 36},
  year={1980},
  publisher = {Academic Press}, address   = {New York}, doi = {10.1016/B978-0-12-044380-2.50007-3}}

@article{Uboe1995II,
  title={A stability property of the stochastic heat equation},
  author={Ubøe, J. and Zhang, T.},
  journal={Stochastic Processes and Their Applications},
  volume={60},
  number={2},
  pages={247 -- 260},
  year={1995},
  publisher={Elsevier}, doi= {10.1016/0304-4149(95)00062-3}}

@article{dobrushin1979,
  title={Non-central limit theorems for non-linear functional of {Gaussian} fields},
  author={R. Dobrushin and P. Major},
  journal={Zeitschrift f{\"u}r Wahrscheinlichkeitstheorie und Verwandte Gebiete},
  volume={50},
  pages={27 -- 52},
  year={1979}, doi={10.1007/BF00535673}}

@article{Taqqu1979,
  title= {Convergence of integrated processes of arbitrary {Hermite} rank},
  author={Taqqu, M.},
  journal={Zeitschrift f{\"u}r Wahrscheinlichkeitstheorie und Verwandte Gebiete},
  volume={50},
  number={1},
  pages={53 -- 83},
  year={1979}, doi= {10.1007/BF00535674}}

@article{Knopova2004,
  author = {V. Knopova},
  title= {Limit behaviour of the renormalized solution to the {Airy} equation with strongly dependent initial data},
  journal      = {Random Operators and Stochastic Equations},
  volume       = {12},
  number       = {1},
  pages        = {35 -- 42},
  year         = {2004}, doi     = {10.1515/156939704323067816}
}

@article{Alodat2020,
author = {Alodat, T. and Leonenko, N. and Olenko, A.},
title = {Limit theorems for filtered long-range dependent random fields},
journal = {Stochastics},
volume = {92},
number = {8},
pages = {1175--1196},
year = {2020},
publisher = {Taylor \& Francis},
doi = {10.1080/17442508.2019.1691211}}

@book{watson1922,
  title={A Treatise on The Theory of {Bessel} Functions},
  author={Watson, G. N.},
  year={1944},
  publisher={Cambridge University Press},
  address   = {Cambridge},
  doi={10.2307/3609752}
}

@article{Olenko2013II,
author = {A. Olenko},
title = {Limit theorems for weighted functionals of cyclical long-range dependent random fields},
journal = {Stochastic Analysis and Applications},
volume = {31},
number = {2},
pages = {199 -- 213},
year = {2013},doi= {10.1080/07362994.2013.741410}}

@book{abramowitz1972,
  title={Handbook of Mathematical Functions with Formulas, Graphs, and Mathematical Tables},
  author={Abramowitz, M. and Stegun, I. A.},
  year={1972},
  publisher={Dover Publications},
  address={New York},
  doi={}
}

@article{kozachenko2020,
    author = {Y. Kozachenko and E. Orsingher and L. Sakhno and O. Vasylyk},
    title = {Estimates for distribution of suprema of solutions to higher-order partial differential equations with random initial conditions},
    journal = {Modern Stochastics: Theory and Applications},
    volume = {7},
    number = {1},
    year = {2020},
    pages = {79--96},
    doi = {10.15559/19-VMSTA146},
}

\end{document}